\documentclass[10pt,a4paper,reqno]{amsart}
\usepackage[applemac]{inputenc}
\usepackage[T1]{fontenc}
\usepackage{hyperref,enumitem}
\setlist{leftmargin=9mm}
\usepackage{amsmath}
\usepackage{amsthm}
\usepackage{amsfonts}
\usepackage{amssymb}
\usepackage{graphicx}
\usepackage{color}
\usepackage{amsbsy}
\usepackage{mathrsfs}
\usepackage{bbm}
\usepackage{xcolor}
\usepackage{dsfont}
	\usepackage{graphicx}
	\usepackage{subcaption}
	
	\newtheorem{theorem}{Theorem}[section]
	\newtheorem{lemma}[theorem]{Lemma}
	
	\newtheorem{proposition}[theorem]{Proposition}
	\newtheorem{corollary}[theorem]{Corollary}
	\newtheorem{example}[theorem]{Example}
	
	\newtheorem{definition}[theorem]{Definition}
	
	\theoremstyle{remark}
	\newtheorem{remark}[theorem]{\it \bf{Remark}\/}
	
	\numberwithin{equation}{section}
	\catcode`@=11
	\def\section{\@startsection{section}{1}%
		\z@{1.5\linespacing\@plus\linespacing}{.5\linespacing}%
		{\normalfont\bfseries\large\centering}}
	\catcode`@=12
	\newcommand{\be}{\begin{equation}}
		\newcommand{\ee}{\end{equation}}
	\newcommand{\bea}{\begin{eqnarray}}
		\newcommand{\eea}{\end{eqnarray}}
	\newcommand{\bee}{\begin{eqnarray*}}
		\newcommand{\eee}{\end{eqnarray*}}

	\def\na{\nabla}

		\def\supp{\textup{supp}\,}

	\def\CC{\mathbb{C}}
	
	\def\RR{\mathbb{R}}
	\def\ZZ{\mathbb{Z}}
	\def\TT{\mathcal{T}}

	\def\ep{\varepsilon}

	\def\calX{{\mathcal X}}

	\def\calX{{\mathcal X}}

	\catcode`@=11
	\def\supess{\mathop{\operator@font Sup\,ess}}
	\catcode`@=12
	\def\CC{\mathbb{C}}

	\def\RR{\mathbb{R}}
	\def\CC{\mathbb{C}}

	\def\ZZ{\mathbb{Z}}
	
	\def\TT{\mathbb{T}}

	\def\omg{\omega}
	\def\nb{\nabla}
	\newcommand{\nrm}[1]{\Vert#1\Vert}

	\usepackage{scalerel,stackengine}
	\stackMath
	\newcommand\reallywidehat[1]{%
		\savestack{\tmpbox}{\stretchto{%
				\scaleto{%
					\scalerel*[\widthof{\ensuremath{#1}}]{\kern-.6pt\bigwedge\kern-.6pt}%
					{\rule[-\textheight/2]{1ex}{\textheight}}
				}{\textheight}%
			}{0.5ex}}%
		\stackon[1pt]{#1}{\tmpbox}%
	}
	
\begin{document}
		
		\title[]{Superlinear gradient growth for 2D Euler equation without boundary}

		\author{In-Jee Jeong}
		\address{Department of Mathematical Sciences and RIM, Seoul National University, 1 Gwanak-ro, Gwanak-gu, Seoul 08826, Republic of Korea.}
		\email{injee$\_$j@snu.ac.kr}

		\author{Yao Yao}
		\address{Department of Mathematics, National University of Singapore, Block S17, 10 Lower Kent Ridge Road, Singapore, 119076, Singapore.}
		\email{yaoyao@nus.edu.sg}

		\author{Tao Zhou}
		\address{Department of Mathematics, National University of Singapore, Block S17, 10 Lower Kent Ridge Road, Singapore, 119076, Singapore.}
		\email{zhoutao@u.nus.edu}

		\date{\today}
		
		\begin{abstract} We consider the vorticity gradient growth of solutions to the two-dimensional Euler equations in domains without boundary, namely in the torus $\mathbb{T}^{2}$ and the whole plane $\mathbb{R}^{2}$. In the torus, whenever we have a steady state $\omega^*$ that is orbitally stable up to a translation  and has a saddle point, we construct $\tilde\omega_0 \in C^\infty(\mathbb{T}^2)$ that is arbitrarily close to $\omega^*$ in $L^2$, such that superlinear growth of the vorticity gradient occurs for an open set of smooth initial data around $\tilde\omega_0$. This seems to be the first superlinear growth result which holds for an open set of smooth initial data (and  does not require any symmetry assumptions on the initial vorticity). Furthermore, we obtain the first superlinear growth result for smooth and compactly supported vorticity in the plane, using perturbations of the Lamb--Chaplygin dipole. 
		\end{abstract}

		\maketitle

		\section{Introduction} 
		
		In this paper, we are concerned with the 2D incompressible Euler equation in vorticity form, when the fluid domain is given by $\Omega= \TT^2$ or $\RR^2$:  \begin{equation}
			\label{eq vorticity form: Euler eq}
			\left\{
			\begin{aligned}
				\partial_t \omg + \mathbf{u} \cdot \na \omg  = 0, \\
			\mathbf{u} = \na^\perp \Delta_{\Omega}^{-1} \omg.
			\end{aligned}
			\right.
		\end{equation} 
	Here $\Delta_\Omega^{-1}$ indicates the inverse Laplacian in $\Omega$. In particular, when $\Omega=\TT^2$, we require the vorticity to be mean-zero in order for the inverse Laplacian to be well-defined. 
	
		Our main results give superlinear-in-time growth of the vorticity gradient $\nrm{ \nb \omg }_{L^\infty}$, for a class of smooth initial data near orbitally stable (up to a translation) steady states with a saddle point. As we shall discuss in more detail below, results on the superlinear growth are quite rare when the fluid domain does not have a boundary. This is why we focus on the domains $\TT^2$ and $\RR^2$ in this work.

		\subsection{Previous works on gradient growth}\label{subsec:literature}
		
		Numerical and experimental studies of the 2D Euler equations \eqref{eq vorticity form: Euler eq} show formation of complex vortex structures and in particular a lot of small scale features (\cite{BoVe}). Given that $\nrm{\omg}_{L^{\infty}}$ is conserved in time, the quantity $\nrm{\nb \omg}_{L^{\infty}}^{-1}$ gives a ``small'' length scale defined by the vorticity: growth of $\nrm{\nb \omg}_{L^{\infty}}$ quantifies creation of small scales in the fluid. 
		
		Classical works of Wolibner \cite{Wo}, H\"older \cite{Hol} and Yudovich \cite{Yud} provide double exponential upper bound of the gradient growth: \begin{equation}\label{eq:expexp}
			\begin{split}
				\nrm{\nb \omg (t ,\cdot)}_{L^{\infty}} \le \left( 1 + \nrm{\nb \omg_{0}}_{L^{\infty}} \right)^{ C\exp(Ct) }
			\end{split}
		\end{equation} with some $C = C( \nrm{\omg_{0}}_{L^\infty} ) > 0$, see \cite{KS} for a simple proof. This bound is valid for general two-dimensional domains, at least when the boundary is sufficiently smooth. Although this allows for a very fast growth of the gradient, there have been relatively few works on establishing lower bounds that diverge as $t \to \infty$. Let us review benchmark works in chronological order and refer to the survey papers (\cite{Ki-sur,Ki-sur2,DE}) for a more comprehensive literature.
		
		\begin{itemize}[leftmargin=\parindent]
			\item Nadirashvili in \cite{Nad} considered \eqref{eq vorticity form: Euler eq} in an annulus, where any constant vorticity state defines a stable steady state whose streamlines are just concentric circles. Then, consider its vorticity perturbation which is not identically equal to a constant on each piece of the annulus boundary. As long as its support has a connected component connecting inner and outer boundaries of the annulus, the vorticity gradient grows at least linearly due to the
			differential rotation speeds of the fluid on the two boundary components. While this linear growth is ``generic'' near the constant vorticity state, it is essential that the vorticity is supported on both parts of the boundary. Similar results were obtained by Yudovich \cite{Yud2}, and see recent works \cite[Section 3]{DE} and \cite{DEJ} for more general results in this direction. 
			
			\smallskip
			\item Denisov in \cite{denisov2009} obtained infinite superlinear growth \eqref{small scale: integral sense} in $\TT^2$, using {stability of the steady state $\cos x_{1} + \cos x_{2}$ by perturbations which are even and $\pi/2$-rotational symmetric at the origin}. This is the first infinite superlinear gradient growth result for \eqref{eq vorticity form: Euler eq};  {more precisely, \cite[Theorem 1.1]{denisov2009} shows the existence of smooth data $\omg_{0}$ on $\mathbb{T}^{2}$ whose solution $\omg(t,x)$ satisfies \begin{equation}\label{eq:den}
		\begin{split}
			\frac{1}{T^{2}}\int_{0}^{T} \Vert \nabla\omg(t,\cdot) \Vert_{L^{\infty}} dt \to \infty \quad\mbox{as}\quad T\to\infty. 
		\end{split}
	\end{equation}}
			
			\smallskip
			\item Kiselev--\v{S}ver\'ak \cite{KS} established that the double exponential upper bound \eqref{eq:expexp} is sharp when the fluid domain is a disc, using initial vorticities which are odd, nonnegative in the half disc, and whose support is attached to the disc boundary. Each of these conditions plays an essential role, and in particular, the odd symmetry is necessary since the proof is based on stability of the ``Bahouri--Chemin'' steady state (odd vorticity which is exactly equal to 1 in the right half of the disc), which is known to hold only within the odd symmetry class. Xu \cite{Xu} extended the double exponential growth to bounded planar domains with an axis of symmetry. Very recently, Zlato\v{s} \cite{zlatos2025} proved that the double-exponential upper bound (with sharp rates) can be achieved in the half plane with compactly supported initial data, where the boundary and odd-in-$x_1$ symmetry are both essential in the proof.
			
			\smallskip
			\item In the torus $\TT^2 = [-\pi,\pi)^2$, Zlato\v{s} \cite{Z} employed a similar strategy as \cite{KS}, and considered perturbations of the Bahouri--Chemin steady state $\omg^*(\mathbf{x}) = \mathrm{sgn}(x_{1})\mathrm{sgn}(x_{2})$ which are odd symmetric with respect to both axes. For the first time, \cite{Z} provided examples of exponential growth of $\nrm{\nb^2\omg(t,\cdot)}_{L^\infty}$ for smooth vorticities and $\nrm{\nb\omg(t,\cdot)}_{L^\infty}$ for $C^{1,\alpha}$ vorticities for any $\alpha<1$. Here the $C^{1,\alpha}$ assumption plays a similar role as the presence of a boundary, as it allows more vorticities to be present in a neighborhood of the hyperbolic stagnation point. 
			
			\smallskip
			\item In the whole plane $\mathbb{R}^2$, for smooth initial vorticity that is compactly supported, so far only linear growth of $\|\nabla\omega\|_{L^\infty}$ was shown in the literature by Choi and the first author \cite{CJ-Lamb}. The proof is based on filamentation from the Lamb--Chaplygin dipole (which we will refer to as the Lamb dipole for simplicity), whose stability was established only recently in \cite{lambdipolesta22}. The observation in \cite{CJ-Lamb} is that if one attaches a trailing ``tail'' to the Lamb dipole, the tail part moves with a strictly slower horizontal velocity relative to the bulk of the vorticity, giving rise to at least linear gradient growth.
		\end{itemize}
		
	To summarize, in all the previous results where the growth rate is faster than linear (e.g. superlinear/exponential/double-exponential), the proofs all impose some symmetry assumptions (odd/even/rotational symmetry) on the initial vorticity. 
Furthermore, in the case of \emph{smooth} vorticity in domains \emph{without  a boundary}, the only superlinear growth result on $\nrm{\nb \omg(t,\cdot)}_{L^\infty}$ is the work of Denisov \cite{denisov2009} in $\mathbb{T}^2$: \eqref{eq:den} is still the fastest known growth rate in the literature. And in the whole plane $\RR^{2}$ case with smooth and compactly supported initial vorticity, there has been no results on growth of $\nrm{\nb \omg(t,\cdot)}_{L^\infty}$ faster than linear.

	\medskip
	In view of this, we would like to investigate the following two questions:
	
	\begin{enumerate}[label=(\alph*)]
	\item It is a natural question how robust the superlinear growth in $\mathbb{T}^2$ is -- the construction by Denisov \cite{denisov2009} requires the initial data being rotational symmetric, where the symmetry assumption is sensitive to small perturbations. Does there exist \emph{an open set of smooth initial data} in certain topology, such that all of them lead to superlinear growth of $\|\nabla\omega(t,\cdot)\|_{L^\infty}$? 
	
	\medskip
	
	\item Is superlinear growth of $\|\nabla\omega(t,\cdot)\|_{L^\infty}$ possible in $\mathbb{R}^2$, for  initial vorticity that is smooth and compactly supported?
	
	\end{enumerate}
	
	In this paper, we give a positive answer to both questions. Below we first discuss our result in $\mathbb{T}^2$ for an open set of initial data, then move on to the superlinear growth in $\mathbb{R}^2$ with odd-in-$x_2$ symmetry.


		\subsection{\text{Small} scale formation on $\TT^2$ without symmetry} 
		In the torus $\TT^2 := [-\pi,\pi)^{2}$, we consider initial data  in the class $C_0^1(\mathbb{T}^2)$, where the subscript $0$ stands for ``mean-zero'':
		\[
		{C}^{1}_0(\mathbb{T}^2) := \Big\{\omega\in C^{1}(\mathbb{T}^2): \int_{\mathbb{T}^2} \omega \, d x = 0\Big\}.
		\]
		Note that it is necessary to impose the mean-zero condition on the torus, to ensure that the stream function $\Phi=\Delta^{-1}\omega$ is well-defined. It is well-known that \eqref{eq vorticity form: Euler eq} is globally well-posed with $C^{1}$ data (\cite{MP,MB,Yud}): for initial data $\omega_0\in C_0^{1}(\mathbb{T}^2)$, there exists a unique global mean-zero solution $\omg$ belonging to $C^{1}_{loc}(  (-\infty,\infty) \times \mathbb{T}^2 )$ with $\omg(t=0)=\omg_{0}$. Furthermore, we shall write $C^\infty_0 = C^1_0 \cap C^\infty$; when $\omg_{0}\in C_{0}^{\infty}$, the corresponding global solution is $C^{\infty}$ in space and time as well.

		Before stating our main result, we introduce a concept of stability in $\TT^2$. 
		
		\begin{definition}[Orbital stability up to a translation]
			\label{def: orbital stable}
			Given  a steady state $\omega^* \in C^1_{0}(\TT^2)$  to the 2D Euler equation \eqref{eq vorticity form: Euler eq}, we say $\omega^*$ is \textbf{orbitally stable up to a translation} if the following is true: for any $\ep >0$, there exists $\delta = \delta(\ep, \omega^*) >0$, such that for any initial data $\omega_0 \in  {C}_0^{1}(\mathbb{T}^2)$ satisfying
			\be
			\| \omega_0 - \omega^* \|_{L^2(\TT^2)} < \delta,
			\label{initial smallness}
			\ee
			there exists some $\mathbf{a}:\mathbb{R}^+\to\mathbb{T}^2$, such that the global-in-time solution $\omega(t,\cdot)$ to \eqref{eq vorticity form: Euler eq} with initial data $\omega_0$ satisfy
			\be
			\| \omega(t, \cdot) - \omega^*(\cdot -\mathbf{a}(t)) \|_{L^2(\TT^2)} < \ep \quad\text{ for all }t\geq 0.
			\label{orbital stability}
			\ee
		\end{definition}
		
		This is the only notion of stability considered in this paper, and the phrase ``up to a translation'' will be sometimes omitted. 
        
        Let us now give the definition of a \emph{saddle point}, where the domain $\Omega$ below can be either $\mathbb{T}^2$ or $\mathbb{R}^2$:
		\begin{definition}
			\label{def: saddle point of w*}
			Assume $\omega^* \in C^1(\Omega)$ is a steady state of \eqref{eq vorticity form: Euler eq}, and $\Phi^*= \Delta_{\Omega}^{-1} \omega^*$ and $\mathbf{u}^*= \na^\perp \Phi^*$ are its stream function and velocity field respectively. Then we say $\mathbf{x}_0 \in \Omega$ is a \emph{saddle point} of the flow $\mathbf{u}^*$ if
			
			\noindent (1) $\mathbf{u}^*(\mathbf{x}_0) =\mathbf{0}$,
			
			\noindent (2) The Hessian matrix $D^2\Phi^*(\mathbf{x}_0)$ has a strictly positive eigenvalue and a strictly negative eigenvalue. \end{definition}
		
		For technical reasons, we will focus on those steady states $\omega^*$ that do not have any smaller period in $\mathbb{T}^2$, i.e.
		\begin{equation}\label{omega_trans}
		\omega^*-\omega^*(\cdot -\mathbf{s})\equiv 0 \text{ in } \mathbb{T}^2~~  \text{ if and only if } ~~ \mathbf{s}=\mathbf{0};
		\end{equation}
		where $\mathbf{s}=\mathbf{0}$ is understood as equivalence in $\mathbb{T}^2$.
		We expect that this assumption can be removed after some extra work, but decide to impose it for the simplicity of proof.
		
		\medskip
		We are now in a position to state our first main result.

		\begin{theorem}[Small scale formation on $\TT^2$ for an open set of initial data]
			\label{thm: torus}
			Assume that $\omega^* \in {C}_{0}^{1}(\TT^2)$ is a steady state of \eqref{eq vorticity form: Euler eq} that is orbitally stable up to translation in the sense of Definition \ref{def: orbital stable}, and satisfies \eqref{omega_trans}. Furthermore, we assume that its flow $\mathbf{u}^*$ has a saddle point in the sense of Definition~\ref{def: saddle point of w*}.    
			
			Then for any $\varepsilon>0$,
			there exist $\tilde\omega_0 \in {C}_{0}^{\infty}(\TT^{2})$ satisfying
			\[
			\| \tilde\omega_0 - \omega^* \|_{L^2(\TT^2)} < \varepsilon \quad \text{ and } \quad   \| \tilde\omega_0 \|_{L^\infty(\TT^2)} \le  \| \omega^* \|_{L^\infty(\TT^2)} + 3
			\]
			and some $C_0(\omega^*), \delta(\ep,\omega^*)>0$, such that for any initial data $\omega_0 \in {C}^{1}_0(\TT^{2})$ with $\| \omega_0 - \tilde\omega_0 \|_{L^\infty(\TT^2)} < \delta$,\footnote{Note that although we only require $\omega_0 $ to be $C^1$, clearly all the mean-zero smooth initial data whose $L^\infty$ distance to $\tilde\omega_0$ is small are included too.} the global classical solution $\omega(t,\cdot)$ to \eqref{eq vorticity form: Euler eq} with initial data $\omega_0$ satisfies
			\be
			\int_0^\infty \| \na \omega(t, \cdot) \|_{L^\infty(\TT^2)}^{-1} dt < C_0, 
			\label{small scale: integral sense}
			\ee
			which in particular implies that
			\be
			\limsup_{t \to \infty} \frac{\| \na  \omega(t, \cdot) \|_{L^\infty(\TT^2)} }{t \log t} = \infty.
			\label{small scale: pw sense}
			\ee
		\end{theorem}
		
		One may wonder whether there exists an orbitally stable steady state (up to a translation) $\omega^*$ with a saddle point. Such steady states do exist, and we give a family of examples in the following.		
		\begin{example}\label{example_cos}
			For any fixed $\alpha, \beta >0$, one can easily check that the function
			\be
			\omega_{\alpha, \beta}^*(\mathbf{x}) := \alpha\cos x_1 + \beta \cos x_2\label{steady solution: TT2}
			\ee
			is a steady state of \eqref{eq vorticity form: Euler eq} on $\TT^2$, and it also satisfies \eqref{omega_trans}.
			\begin{itemize}
			\item 
			\textbf{Existence of saddle points:} Its corresponding stream function is given by $\Phi_{\alpha, \beta}^*(\mathbf{x}) := -\alpha\cos x_1 - \beta \cos x_2$, and its flow has a saddle point at $(\pi,0)$ in the sense of Definition~\ref{def: saddle point of w*}. 
			\item
			\textbf{Orbital stability:} Wang--Zuo \cite[Theorem 1.4(iii)]{WZ} showed that $\omega_{\alpha,\beta}^*$ is orbitally stable up to a translation, in the sense of Definition \ref{def: orbital stable}. Elgindi \cite[Theorem 1.3]{elgindi2023stabi} further obtained quantitative bounds on how $\delta$ depends on $\ep$ in Definition~\ref{def: orbital stable}. 
			\end{itemize}
		\end{example}
		
Next, let us discuss the similarities and differences between the proof strategies of Denisov \cite{denisov2009} and Theorem~\ref{thm: torus}. In \cite{denisov2009}, the superlinear growth is obtained by the philosophy that \[
\text{``orbital stability + saddle point $\Longrightarrow$ superlinear growth''},
\] where the construction was based on the orbital stability of steady state {$\omega_{1,1}^*$} under {perturbations which are \emph{even and $\pi/2$-rotational symmetric} at $(0,0)$. These symmetry conditions propagate in time; namely, if the initial data $\omg_{0}$ satisfies $\omg_{0}(\mathbf{x}) = \omg_{0}(-\mathbf{x})$ and $\omg_{0}(\mathbf{x}) = \omg_{0}(\mathbf{x}^\perp)$ (with $\mathbf{x}^\perp = (-x_2,x_1)$), then the solution $\omega(t,\cdot)$ satisfies these conditions for all $t$. This in particular} ensures that $\omega(t,\cdot)$ stays close to $\omega_{1,1}^*$ for all times (without the need to introduce any translation), so one can find a (fixed) small rectangle centered at {$(\pi,0)$} such that $\mathbf{u}(t,\cdot)$ has strictly outward flux on a pair of opposite sides and strictly inward flux on the other two sides.\footnote{Once such a rectangle with inward/outward flux is found, \cite{denisov2009} used a simple but clever argument to show that it leads to superlinear gradient growth if there are level sets of different values going across the rectangle boundaries -- we sketch a slightly general version of this argument in Appendix~\ref{appendix A}.}

In Theorem~\ref{thm: torus}, our main contribution is that \eqref{small scale: integral sense} is obtained \emph{without any symmetry assumptions} on the initial vorticity, thereby obtaining superlinear growth for an \emph{open set} in $L^{\infty}$ near $\tilde\omega_0$. 
While the framework of our proof follows a similar philosophy
\[
\text{``orbital stability \emph{up to a translation} + saddle point $\Longrightarrow$ superlinear growth''},
\]  the lack of symmetry assumptions gives rise to some substantial technical difficulty: the best we can hope for is orbital stability \emph{up to a translation}, where $\omega(t,\cdot)$ stays close to $\omega^*(\cdot-\mathbf{a}(t))$ for some unknown ``shift vector'' $\mathbf{a}(t)$ from Definition \ref{def: orbital stable}. Therefore, one cannot expect the flux along the boundaries of a fixed rectangle to have a certain sign for all times, since $\mathbf{u}(t,\cdot)$ is close to $\mathbf{u}^*(\cdot-\mathbf{a}(t))$ which is drifting in time.
	
A natural idea is to consider the equation in a moving frame centered at $\mathbf{a}(t)$, but this is difficult due to the lack of differentiability of $\mathbf{a}(t)$ (note that $\mathbf{a}(t)$ is not assumed to have any regularity properties by definition). To overcome this issue, a key step in the proof is to introduce an ``approximate translation vector'' $\mathbf{p}(t)$ such that $\omega(\cdot,t)$ stays close to $\omega^*(\cdot-\mathbf{p}(t))$ for all time, where the speed of $\mathbf{p}(t)$ can be made arbitrarily small if $\omega_0$ is sufficiently close to $\omega^*$; see Proposition~\ref{prop_p} for the precise statement. We can then consider the equation in a moving frame centered at $\mathbf{p}(t)$ (which is drifting very slowly) to obtain sign conditions on the flux along the sides of a parallelogram, which leads to superlinear growth of the gradient.

Lastly, we remark that our \eqref{small scale: integral sense} implies \eqref{eq:den} (by applying H\"older's inequality on the time interval $[T,2T]$ and then taking $T\to\infty$). This slight improvement comes from our technical lemma proved in the Appendix (see Lemma \ref{grad_growth}).

		\subsection{Small scale formation on $\RR^2$}
		 When the fluid domain is $\mathbb{R}^2$ and the vorticity is compactly supported, the enemy in obtaining small scale creation is that an initially compactly supported vorticity could ``spread out'' in space and escape to infinity as $t\to\infty$, slowing down any mechanisms of gradient growth.  For this reason, instead of working with initial data with odd-odd symmetry as in \cite{Z, zlatos2025}, we will consider initial data that is a small perturbation of the Lamb dipole $\omega_L$ defined in \eqref{Lamb dipole: expression}.

			\begin{figure}[htbp]
	\begin{center}
	\includegraphics[scale=1]{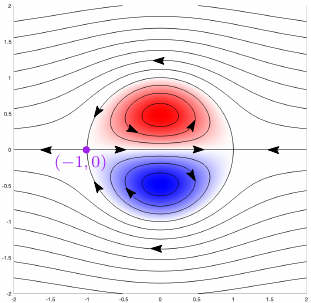}
	\caption{\label{fig_lamb} Illustration of the Lamb dipole $\omega_L$ and its streamlines in the moving frame with velocity 1. The velocity field in the moving frame has two saddle points at $(\pm 1,0)$.}
	\end{center}
	\end{figure}
		 
		 The Lamb dipole is a traveling wave solution of \eqref{eq vorticity form: Euler eq} with velocity 1 towards the right; see Section~\ref{subsec_lamb} for a quick review on its basic properties. In the moving frame, its velocity field has two saddle points at $(\pm 1, 0)$; see Figure~\ref{fig_lamb} for an illustration. 
The orbital stability of $\omega_L$ is known under some symmetry and sign conditions. To make it precise, we define the space $\calX$ and a norm by
		\begin{equation}	\label{norm: X}
			\begin{split}
						\calX : = \Big\{ \omega \in L^1 \cap L^\infty(\RR^2): \int_{\RR^2} |\mathbf{x}| |\omega(\mathbf{x})| \, d\mathbf{x} < \infty \Big\}, \quad 	\| \omega \|_{\calX} := \| x_2 \omega \|_{L^1(\RR^2)} + \| \omega \|_{L^2(\RR^2)}, 
						\end{split}
						\end{equation}
						and then define the subclass $\calX_{odd,+}$ of $\calX$ by \begin{equation}	\label{def: Xodd+}
			\begin{split}
				 \calX_{odd,+} := \big\{ \omega \in \calX \, : \, \omg(\mathbf{x}) = -\omg(x_{1},-x_{2}), \, \omg \ge 0 \text{ for } x_2 >0 \big\}.
			\end{split}
		\end{equation}
	 Abe--Choi \cite{lambdipolesta22} showed that $\omega_L$ is orbitally stable in $\calX_{odd,+}$ up to a translation; see Proposition~\ref{prop: orbital stability of Lamb dipole} for the precise statement.

		\medskip
	 
	 We are now ready to state our main result in $\mathbb{R}^2$:
		\begin{theorem}[Small scale formation on $\RR^2$]
			\label{thm: small scale R2}
			Consider \eqref{eq vorticity form: Euler eq} on the whole space $\RR^2$. For any $\ep > 0$, there exist $\delta > 0, C>0$ and $\tilde\omega_0 \in C^\infty(\RR^2) \cap \calX_{odd,+}$ satisfying 			
			\[
			\| \tilde\omega_0 - \omega_L \|_{\calX} \leq\ep, \quad |\supp\tilde\omega_0|\leq 4,\quad \|\tilde\omega_0\|_{L^\infty} \leq \|\omega_L\|_{L^\infty}+2, 
			\]  
			such that for any initial data $\omega_0 \in C^1(\RR^2) \cap \calX_{odd,+}$ satisfying
			\[
			\| \omega_0 - \tilde\omega_0 \|_{\calX} \leq\delta,  \quad |\supp\omega_0|\leq 5,\quad \|\omega_0-\tilde\omega_0\|_{L^\infty} \leq 1,
			\]
			 the solution $\omg(t,\cdot)$ to \eqref{eq vorticity form: Euler eq} with initial condition $\omg_0$ satisfies \begin{equation*}
				\begin{split}
					\int_0^\infty \| \na \omg(t, \cdot) \|_{L^\infty(\RR^2)}^{-1} dt < C \quad \mbox{and} \quad 	\limsup_{t \to \infty} \frac{\| \na  \omg(t, \cdot) \|_{L^\infty(\RR^2)} }{t \log t} = \infty.
				\end{split}
			\end{equation*}  	\end{theorem}
		
		
		\begin{remark}
			Unlike the theorem on $\TT^2$, Theorem~\ref{thm: small scale R2} requires odd symmetry and sign assumptions on the initial vorticity. These restrictions  directly come from the stability result for $\omg_L$. Still, to the best of our knowledge, this is the first superlinear gradient growth result for smooth and compactly supported vorticity in $\RR^2$. 
		\end{remark}
		
		Theorem~\ref{thm: small scale R2} is proved by constructing small perturbation (in $L^2$ sense) to $\omega_L$, such that $\omega_0$ is smooth, and $\omega_0\geq 1$ and $\leq -1$ along the two horizontal line segments of length order 1, illustrated in Figure~\ref{fig_lamb_dynamics}(a). In a suitable moving frame centered at $(t+p(t))\mathbf{e}_1$, we will show that the solution stays close to $\omega_L$ for all times, and $\dot p(t)$ (which is the artificial velocity induced by $p(t)$ in the moving frame) can be made arbitrarily small. This argument and the choice of $p(t)$ is done in Section~\ref{subsec_p}, and it is the heart of the proof.\footnote{ Note that the approach we used in Section~\ref{sec_trans_torus}  is very specific to the torus setting, and cannot be adapted to the whole space case. In Section~\ref{subsec_p} we also expain why one cannot define $p(t)$ by tracking the center of mass of $\omega(t,\cdot)1_{\{x_2>0\}}$ either. At the end, what works for us is to define $p(t)$ in an implicit way.}
		
		Once this is done, the velocity field in this moving frame can be shown to have positive/negative flux along the sides of a small square centered at $(-1,0)$; see Figure~\ref{fig_lamb_dynamics}(b) for an illustration. One can then focus on the dynamics in the square and obtain superlinear growth of the gradient with a similar argument as in \cite{denisov2009}.

			\begin{figure}[htbp]
	\begin{center}
	\includegraphics[scale=1]{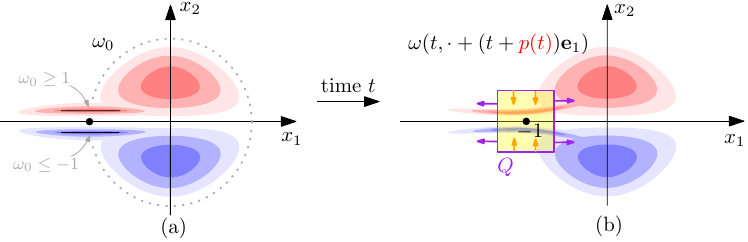}
	\caption{\label{fig_lamb_dynamics} (a) Illustration of the initial data $\omega_0$ in Theorem~\ref{thm: small scale R2}. (b) In some suitable moving frame centered at $(t+p(t))\mathbf{e}_1$, the solution remains close to $\omega_L$ for all times, and the velocity field (in the moving frame) has strictly positive/negative flux along the boundary of a small square $Q$ centered at $(-1,0)$: the signs of flux are shown by the purple and orange arrows.} \end{center}
	\end{figure}

\medskip
	Lastly, we would like to remark that our method does not rely on the Lamb dipole in an essential way. If there is another relative equilibria $\omega_s$ that is a traveling wave  or uniformly-rotating solution to \eqref{eq vorticity form: Euler eq}, and satisfies both of the following:
	
	(1) $\omega_s$ is known to be orbitally stable up to a translation (or rotation); 
	
	(2) In the moving (or rotating) frame,  its stream function has a saddle point,
	
\noindent then it is likely that the proof can be adapted to initial condition near $\omega_s$. 
	The issue is that there are not many known examples of such orbitally stable relative equilibria. For instance, orbital stability of the Kirchhoff ellipses is known when the aspect ratio of the ellipse is between $1/3$ and $3$ (\cite{Tang,Wan,WP85}), and there are two saddle points in the corotating reference frame. However, this stability statement requires that the support of the perturbation lies in some ball containing the ellipse, an assumption which is not known to propagate for all times.  
	
	\subsection{Organization of the paper} In Section~\ref{sec: torus} we focus on the domain $\mathbb{T}^2$ and prove Theorem~\ref{thm: torus}. In Section \ref{sec_R} we work with the domain $\mathbb{R}^2$ and prove Theorem~\ref{thm: small scale R2}. 
These two sections are logically independent of each other. 

In each section, the key argument is to construct a suitable moving frame, such that the solution stays close to $\omega^*$ or $\omega_L$ for all times, and the velocity field in the moving frame has strictly positive/negative flux along the sides of some parallelogam/square. Here the arguments in $\mathbb{T}^2$ and $\mathbb{R}^2$ are completely different: see Section~\ref{sec_trans_torus} for the torus case and Section~\ref{subsec_p} for the plane case. 

The rest of the proofs are similar for $\mathbb{T}^2$ and $\mathbb{R}^2$; in particular, in both parts, we used the argument  ``signed flux leads to superlinear growth of gradient'' in \cite{denisov2009}, which we state and prove in Appendix~\ref{appendix A} for the sake of completeness. 
		
		\subsection{Notations} 
		
		\begin{itemize}
			\item We use bold letters to denote vectors in $\TT^2$, $\RR^2$ and $\ZZ^2$.
			\item 
			For $x,y\in\mathbb{T}$, we denote by $|x-y|:=\text{dist}(x,y)$ the smallest distance (in $\mathbb{R}$) between  two numbers in equivalent classes of $x$ and $y$ respectively. Likewise, for $\mathbf{x},\mathbf{y}\in\mathbb{T}^2$, we denote by $|\mathbf{x}-\mathbf{y}|:=\text{dist}(\mathbf{x},\mathbf{y})$ the smallest distance (in $\mathbb{R}^2$) between  two numbers in equivalent classes of $\mathbf{x}$ and $\mathbf{y}$ respectively. 
		\end{itemize}
		
		\subsection*{Acknowledgments} IJ has been supported by the NRF grant from the Korea government (MSIT), No. 2022R1C1C1011051, RS-2024-00406821. YY has been supported by the NUS startup grant, MOE Tier 1 grant, and the Asian Young Scientist Fellowship. TZ has been supported by the MOE Tier 1 grant.

\section{Superlinear growth on the torus for an open set of initial data}
\label{sec: torus}

Throughout this section, we assume $\omega^* \in C_0^1(\mathbb{T}^2)$ is a steady state of \eqref{eq vorticity form: Euler eq} satisfying Definition~\ref{def: orbital stable}, and also has a saddle point. By Definition~\ref{def: orbital stable}, if $\omega_0$ is close to $\omega^*$, \eqref{orbital stability}  ensures that $\omega(t,\cdot)$ remains close to some translation $\omega^*(\cdot-\mathbf{a}(t))$ for some $\mathbf{a}(t):\mathbb{R}^+\to\mathbb{T}^2$, but it does not give any regularity of the translation vector $\mathbf{a}(t)$. In particular, the optimal translation vector given by
$
\mathbf{a}(t) := \text{argmin}_\mathbf{a\in\mathbb{T}^2} \|\omega(t,\cdot)-\omega^*(\cdot-\mathbf{a})\|_{L^2}
$
 could be discontinuous in time. 

\subsection{Constructing an approximate translation vector $\mathbf{p}(t)$}
\label{sec_trans_torus}
The key result of our proof relies on a careful choice of another translation vector $\mathbf{p} \in C^1(\mathbb{R}^+; \mathbb{T}^2)$ (which we call the ``approximate translation vector''), which satisfies both of the following properties:

(1) the orbital stability result \eqref{orbital stability} still holds when we replace the translation vector $\mathbf{a}(t)$ by $\mathbf{p}(t)$, at the expense that the error gets at most inflated by a constant factor only depending on $\omega^*$;

(2) The speed of $\mathbf{p}$ can be made arbitrarily small\footnote{As we will see in later subsections, the smallness of the translation speed $\dot{\mathbf{p}}$ is crucial for the proof, since we will prove the gradient growth by considering the evolution in a moving frame with translation vector $\mathbf{p}(t)$.
} by choosing $\omega_0$ sufficiently close to $\omega^*$.

Namely, we aim to prove the following result:

\begin{proposition}\label{prop_p}
			Assume that $\omega^* \in {C}_{0}^{1}(\TT^2)$ is a steady state of \eqref{eq vorticity form: Euler eq} that is orbitally stable up to translation in the sense of Definition \ref{def: orbital stable}, and satisfies \eqref{omega_trans}. Furthermore, we assume its flow $\mathbf{u}^*$ has a saddle point in the sense of Definition~\ref{def: saddle point of w*}.    
Then there exist constants $C(\omega^*), \ep_0(\omega^*)>0$ only depending on $\omega^*$, such that the following holds: 

For any $\ep\in (0,\ep_0)$, there exists some $\delta_0=\delta_0(\ep,\omega^*)>0$, such that for all $\omega_0\in C_0^1(\mathbb{T}^2)$ satisfying $\|\omega_0-\omega^*\|_{L^2(\mathbb{T}^2)} \leq \delta_0$, there exists some $\mathbf{p}\in C^1(\mathbb{R}^+; \mathbb{T}^2)$ depending on $\omega_0$, such that
\begin{equation}\label{goal_dist}
\|\omega(t,\cdot) - \omega^*(\cdot - \mathbf{p}(t))\|_{L^2(\mathbb{T}^2)} \leq C(\omega^*)\ep\quad\text{ for all }t\geq 0,
\end{equation}
and
\begin{equation}\label{goal_deriv}
\left|\frac{d}{dt} \mathbf{p}(t)\right| \leq C(\omega^*)\ep \quad\text{ for all }t\geq 0.
\end{equation}
\end{proposition}


\medskip

Next we discuss how to define $\mathbf{p}(t)$.
 Since our spatial domain is the torus $\mathbb{T}^2$, a natural idea is to consider the phase\footnote{For any non-zero complex number $z=r e^{i\theta}$, we define $\theta := \arg z \in\mathbb{T}$ to be its phase.} of the first Fourier coefficients of $\omega(t,\cdot)$ and $\omega^*$, and define $\mathbf{p}(t) = (p_1(t),p_2(t))^T$, where $p_j(t)$ is given by
\[
p_j(t) =  \arg \hat\omega^*(\mathbf{e}^{(j)})-\arg \hat\omega(t,\mathbf{e}^{(j)}) \quad\text{ for } j=1,2,
\]
where $\mathbf{e}^{(1)} := (1,0)^T$ and $\mathbf{e}^{(2)} := (0,1)^T$. However, there is potentially an issue with such definition: $\arg \hat\omega^*(\mathbf{e}^{(j)})$ is not well-defined if $\hat\omega^*(\mathbf{e}^{(j)})$ happens to be zero.

To solve this issue, we will define $\mathbf{p}(t)$ using other Fourier modes instead of the first modes. Towards this end, we first make a simple observation for $\omega^*$: if it has a saddle point, then there must be two linearly independent Fourier modes $\mathbf{k}^{(1)}, \mathbf{k}^{(2)} \in \mathbb{Z}^2$ at which $\hat \omega^*$ does not vanish.

\begin{lemma}
\label{lem_non_zero_fourier}
Assume $\omega^*$ is a steady state of \eqref{eq vorticity form: Euler eq} satisfying Definition~\ref{def: orbital stable}, and its flow has a saddle point in the sense of Definition \ref{def: saddle point of w*}. Then there exist two linearly independent vectors $\mathbf{k}^{(1)}, \mathbf{k}^{(2)} \in \ZZ^2$, such that 
\[\hat \omega^*(\mathbf{k}^{(j)}) : = \int_{\TT^2} e^{-i\mathbf{k}^{(j)} \cdot \mathbf{x}} \omega^*(\mathbf{x}) d\mathbf{x}\not= 0 \quad\text{ for } j=1,2.
\]

\end{lemma}

\begin{proof}
Since $\omega^*$ is mean-zero but not constant zero, there exists some $0 \not = \mathbf{k}^{(1)} \in \ZZ^2$ such that $\hat \omega^*(\mathbf{k}^{(1)}) \neq 0$. 

Towards a contradiction, suppose there does not exist a $\mathbf{k}^{(2)}$ that is linearly independent with $\mathbf{k}^{(1)}$ such that $\hat \omega^*(\mathbf{k}^{(2)}) \neq 0$. Then  $\hat \omega^*$ can only possibly be non-zero at multiples of $\mathbf{k}^{(1)}$. In other words, there exists a sequence $\{ d_j \}_{j\in\mathbb{Z}} \subset \CC$, such that $\omega^*$ can be represented by
\begin{equation}\label{temp0}
\omega^*(\mathbf{x}) = \sum_{ 0\not= j \in \ZZ} d_j e^{ i j \mathbf{k}^{(1)} \cdot \mathbf{x}} \quad\text{for all } \mathbf{x} \in \TT^2,
\end{equation}
where the sum does not include the $j=0$ term since $\omega$ is mean-zero. Using the fact that the stream function $\Phi^*(\mathbf{x})$ solves $\Delta \Phi^* = \omega^*$ on $\TT^2$, $\Phi^*$ can be written as
\begin{equation}\label{tempPhi}
\Phi^*(\mathbf{x}) = \sum_{0 \not= j \in \ZZ} \Phi_j e^{ij \mathbf{k}^{(1)} \cdot \mathbf{x}} \quad \text{for all } \mathbf{x} \in \TT^2,
\end{equation}
with $\Phi_j := -\frac{d_j}{ j^2 |\mathbf{k}^{(1)} |^2} $. Let $\mathbf{k}_\perp^{(1)}$ be a unit vector satisfying $\mathbf{k}_\perp^{(1)} \cdot \mathbf{k}^{(1)}=0$, then for any $\mathbf{x} \in \TT^2$ and $b\in\mathbb{R}$, \eqref{tempPhi} implies 
\[
\Phi^*\left(\mathbf{x} + b \mathbf{k}_\perp^{(1)} \right) = \sum_{0 \not= j \in \ZZ} \Phi_j e^{i j \mathbf{k}^{(1)} \cdot \left(\mathbf{x}+b \mathbf{k}_{\perp}^{(1)} \right)}  = \sum_{0 \not= j \in \ZZ} \Phi_j e^{ i j \mathbf{k}^{(1)} \cdot \mathbf{x}}  = \Phi^*(\mathbf{x}),
\]
i.e. all straight lines parallel to $\mathbf{k}_\perp^{(1)}$ are level sets of $\Phi^*$. This implies the Hessian matrix of $\Phi^*$ is degenerate at any $\mathbf{x} \in \TT^2$, contradicting the assumption that there is a saddle point. Hence we have verified the existence of linearly independent $\{\mathbf{k}^{(1)}, \mathbf{k}^{(2)}\} \subset \ZZ^2$  such that $\hat \omega^*{(\mathbf{k}^{(j)})} \neq 0$ for $j=1,2$.
\end{proof}

From now on, let us fix two linearly independent vectors $\mathbf{k}^{(1)}, \mathbf{k}^{(2)}\in \mathbb{Z}^2$ as given by Lemma~\ref{lem_non_zero_fourier}. We will define $\mathbf{p}(t):\mathbb{R}^+\to\mathbb{T}^2$ to be a differentiable function such that
\[
\arg \reallywidehat{\omega^*(\cdot-\mathbf{p}(t))}(\mathbf{k}^{(j)}) = \arg\hat \omega(t,\mathbf{k}^{(j)}) \quad\text{ for } j=1,2, ~t\geq 0.
\]
Since the left hand side is equal to $\arg \hat \omega^*(\mathbf{k}^{(j)})-\mathbf{k}^{(j)}\cdot \mathbf{p}(t)$, we can rewrite the above equation into a linear equation in $\mathbb{T}^2$:
\begin{equation}\label{eq_matrix}
K\mathbf{p}(t) = \mathbf{b}(t) \quad\text{ for } t\geq 0,
\end{equation}
where 
\begin{equation}\label{def_Kb}
K := \left(
	\begin{matrix}
		k_1^{(1)} & k_2^{(1)}   \vspace{0.2cm}
        \\
		k_1^{(2)} & k_2^{(2)}
	\end{matrix}
	 \right) \quad\text{and}\quad 
\mathbf{b}(t) := \left(\begin{matrix}  \arg\hat \omega^*(\mathbf{k}^{(1)})-\arg\hat \omega(t,\mathbf{k}^{(1)})\\ 
\arg\hat \omega^*(\mathbf{k}^{(2)})-\arg\hat \omega(t,\mathbf{k}^{(2)})\end{matrix}\right).
	 \end{equation}

Below let us discuss the solutions to \eqref{eq_matrix}:
		\begin{itemize}[leftmargin=\parindent]
\item \textbf{Existence of a solution in $C^1(\mathbb{R}^+;\mathbb{T}^2)$.} Note that $K$ is invertible by Lemma~\ref{lem_non_zero_fourier}; also, $\mathbf{b}\in C^1(\mathbb{R}^+;\mathbb{T}^2)$ since $\omega$ is $C^1$ in $t$. These imply the existence of some $\mathbf{p}\in C^1(\mathbb{R}^+; \mathbb{T}^2)$ solving \eqref{eq_matrix}: for example, one way to construct such a solution is to first lift $\mathbf{b}$ to $\mathbb{R}^2$ by defining $\mathbf{\tilde b} \in C^1(\mathbb{R}^+;\mathbb{R}^2)$ such that $\Pi \mathbf{\tilde b} = \mathbf{b}$ (where $\Pi$ is the projection from $\mathbb{R}^2$ to $\mathbb{T}^2$). We can then solve \eqref{eq_matrix} in $\mathbb{R}^2$ (with right hand side given by $\mathbf{\tilde b}$) to obtain $\mathbf{\tilde p} := K^{-1} \mathbf{\tilde b} \in C^1(\mathbb{R}^+;\mathbb{R}^2)$, and then project it to $\mathbb{T}^2$ by setting $\mathbf{p}(t) := \Pi \mathbf{\tilde p}(t)$.

\medskip
\item \textbf{Number of solutions.} Since \eqref{eq_matrix} is an equation in $\mathbb{T}^2$, there may be more than one solutions at each $t\geq 0$. Namely,
we define the set $S\subset\mathbb{T}^2$ by
\begin{equation}\label{def_S}
S:=\{\mathbf{s}\in \mathbb{T}^2: K\mathbf{s}=\mathbf{0}\},
\end{equation}
where the $\mathbf{0}$ on the right hand side is a vector in $\mathbb{T}^2$. Note that $S$ is non-empty (since $\mathbf{0}\in S$), and it contains finitely many elements since $K$ is an invertible matrix with integer entries. Denote the cardinality of $S$ by $N$, and its elements by $\mathbf{s}^{(1)},\cdots, \mathbf{s}^{(N)}$. Then for any $\mathbf{p} \in C^1(\mathbb{R}^+; \mathbb{T}^2)$ solving \eqref{eq_matrix}, $\mathbf{p}+\mathbf{s}^{(n)}$ is also a solution for $n=1,\dots, N$, and these are all the solutions to \eqref{eq_matrix} in $C^1(\mathbb{R}^+; \mathbb{T}^2)$.

\end{itemize}

 The above discussion yields a total of $N$ solutions in $C^1(\mathbb{R}^+; \mathbb{T}^2)$ to \eqref{eq_matrix}. Now we are ready to define the \emph{approximate translation vector} $\mathbf{p}(t)$ as follows:

\begin{definition}\label{def_p}
Under the assumptions of Proposition~\ref{prop_p}, let us define the \emph{approximate translation vector} $\mathbf{p}:\mathbb{R}^+\to\mathbb{T}^2$ to be the solution to \eqref{eq_matrix} in $C^1(\mathbb{R}^+; \mathbb{T}^2)$, such that $|\mathbf{p}(0)-\mathbf{a}(0)|$ is the smallest among the $N$ solutions, where $\mathbf{a}(0)$ is given by Definition~\ref{def: orbital stable} (for the particular initial data of $\omega_0$).
\end{definition}

\subsection{Proof of Proposition~\ref{prop_p}}

Throughout the proof, let $\mathbf{p}(t)$ be given as in Definition~\ref{def_p} for $t\geq 0$. Since $\omega^*$ is orbitally stable up to a translation, by Definition~\ref{def: orbital stable}, there exists $\mathbf{a}:\mathbb{R}^+\to\mathbb{T}^2$ such that
\begin{equation}\label{eq_stab}
\|\omega( t,\cdot) - \omega^*(\cdot-\mathbf{a}(t))\|_{L^2} \leq \ep \quad\text{ for all }t\geq 0.
\end{equation}

Below we state the proofs of \eqref{goal_dist} and \eqref{goal_deriv}, which are independent from each other.

\begin{proof}[\textbf{\textup{Proof of \eqref{goal_dist}}}]
By Parseval's identity, \eqref{eq_stab} implies
\begin{equation}\label{parseval}
\sum_{j=1,2} \Big|\hat\omega(t,\mathbf{k}^{(j)}) - \hat\omega^*(\mathbf{k}^{(j)}) e^{-i \mathbf{k}^{(j)}\cdot \mathbf{a}(t)}\Big|^2 \leq 4\pi^2 \ep^2 \quad\text{ for all }t\geq 0.
\end{equation}
Recall that for any two complex numbers $z,w\in \mathbb{C}$ with $z\neq 0$ and $|w-z|\leq |z|/2$, we have the elementary inequality 
\[
|\arg w - \arg z| \leq 2|\sin(\arg w - \arg z)| \leq \frac{2|z-w|}{|z|}.
\]
Applying this to \eqref{parseval} with $w=\hat\omega(t,\mathbf{k}^{(j)})$ and $z=\hat\omega^*(\mathbf{k}^{(j)}) e^{-i \mathbf{k}^{(j)}\cdot \mathbf{a}(t)}$ yields
\[
\sum_{j=1,2} \frac{|\hat \omega^*(\mathbf{k}^{(j)})|^2}{4} \left|\arg \hat\omega(t,\mathbf{k}^{(j)}) - \arg \hat\omega^*(\mathbf{k}^{(j)}) + \mathbf{k}^{(j)}\cdot \mathbf{a}(t)\right|^2 \leq 4\pi^2 \ep^2.
\]
Let us define $c(\omega^*):= \min\{|\hat \omega^*(\mathbf{k}^{(1)})|, |\hat \omega^*(\mathbf{k}^{(2)})|\} >0$. Using the definition of $K$ and $\mathbf{b}$ in \eqref{def_Kb} together with the definition of $c(\omega^*)$, the above inequality can be written into matrix form as
\[
|K\mathbf{a}(t) - \mathbf{b}(t)| \leq  4\pi c(\omega^*)^{-1}\ep,
\]
and replacing $\mathbf{b}(t)$ by $K\mathbf{p}(t)$ (by definition of $\mathbf{p}(t)$) yields
\[
|K(\mathbf{a}(t)-\mathbf{p}(t))| \leq 4\pi c(\omega^*)^{-1}\ep.
\]
Using the definition of $S$ in \eqref{def_S},  the above inequality implies  
\begin{equation}\label{dist_small}
\text{dist} (\mathbf{a}(t)-\mathbf{p}(t), S) \leq \underbrace{4\pi c(\omega^*)^{-1}\|K\|^{-1}}_{=: C_1(\omega^*)} \ep \quad\text{ for all }t\geq 0,
\end{equation}
where $\|K\|$ is the spectral radius of the matrix $K$.
For any $t\geq 0$, let us choose $\mathbf{s}(t)\in S$ such that 
\begin{equation}\label{temp0}
|\mathbf{a}(t)-\mathbf{p}(t)-\mathbf{s}(t)| = \text{dist} (\mathbf{a}(t)-\mathbf{p}(t), S).
\end{equation}
Since $S$ only contains finitely many elements, we can set $\ep_0(\omega^*)$ sufficiently small such that $2C_1(\omega^*)\ep_0(\omega^*) \leq \min_{i\neq j}|\mathbf{s}^{(i)}-\mathbf{s}^{(j)}|$. For any $\ep\in(0,\ep_0(\omega^*))$, $\mathbf{s}(t) \in \TT^2$ is uniquely defined for each $t\geq 0$. At $t=0$, since $\mathbf{p}(0)$ is chosen as the solution to $K\mathbf{p}(0)=\mathbf{b}(0)$ that minimizes $|\mathbf{a}(0)-\mathbf{p}(0)|$, we know that $\mathbf{s}(0)=\mathbf{0}$.

By definition, $\mathbf{s}(t)$ may jump across different elements in $S$. Our goal is to show that this is impossible, and we actually have $\mathbf{s}(t)\equiv\mathbf{0}$ for all $t\geq 0$.

Since $\omega^* \in C^1(\mathbb{T}^2)$, for any $t\geq 0$ we have
\[
\begin{split}
\|\omega^*(\cdot - \mathbf{a}(t))-\omega^*(\cdot - \mathbf{p}(t)-\mathbf{s}(t))\|_{L^2} &\leq 2\pi \|\omega^*(\cdot - \mathbf{a}(t))-\omega^*(\cdot - \mathbf{p}(t)-\mathbf{s}(t))\|_{L^\infty} \\
&\leq 2\pi \|\nabla \omega^*\|_{L^\infty} |\mathbf{a}(t)-\mathbf{p}(t)-\mathbf{s}(t)| \\
&\leq  2\pi \|\nabla \omega^*\|_{L^\infty} C_1(\omega^*)\ep,
\end{split}
\]
where the last inequality follows from \eqref{dist_small} and \eqref{temp0}. 
Putting this  together with \eqref{eq_stab} yields
\begin{equation}\label{temp1}
\|\omega( t,\cdot)-\omega^*(\cdot - \mathbf{p}(t)-\mathbf{s}(t))\|_{L^2} < C_2(\omega^*)\ep \quad\text{ for all }t\geq 0,
\end{equation}
where $C_2(\omega^*) := 2\pi \|\nabla \omega^*\|_{L^\infty} C_1(\omega^*)+1$. 

For  two arbitrary times $\tau, t\in[0,\infty)$, applying \eqref{temp1} at these two times and using the triangle inequality, we have
\begin{equation}
\label{temp2}
\|\omega(\tau,\cdot) - \omega( t,\cdot) - (\omega^*(\cdot - \mathbf{p}(\tau)-\mathbf{s}(\tau))-\omega^*(\cdot - \mathbf{p}(t)-\mathbf{s}(t)))\|_{L^2} \leq 2 C_2(\omega^*)\ep.
\end{equation}
As we take the limit $\tau\to t$, recall that $\lim_{\tau\to t}\|\omega(\tau,\cdot)-\omega(t,\cdot)\|_{L^2}=0$ since $\omega$ is continuous in time, and $\lim_{\tau\to t} \mathbf{p}(\tau) = \mathbf{p}(t)$ since $\mathbf{p}\in C^1(\mathbb{R}^+,\mathbb{T}^2)$. Combining these with \eqref{temp2} gives
\[
\limsup_{\tau\to t}\|\omega^*(\cdot - \mathbf{p}(t)-\mathbf{s}(\tau))-\omega^*(\cdot - \mathbf{p}(t)-\mathbf{s}(t))\|_{L^2} \leq 2 C_2(\omega^*)\ep,
\]
which is equivalent with the following (after we take the same translation $\mathbf{p}(t)$ to the two functions):
\begin{equation}\label{temp3}
\limsup_{\tau\to t}\|\omega^*(\cdot -\mathbf{s}(\tau))-\omega^*(\cdot -\mathbf{s}(t))\|_{L^2} \leq 2 C_2(\omega^*)\ep \quad\text{ for all }t\in[0,\infty).
\end{equation}
By the assumption \eqref{omega_trans}, $\omega^*$ does not stay the same if we translate it by any vector $\mathbf{s}\in S$ that is not zero. This leads to $c(\omega^*) := \min_{\mathbf{0} \not =\mathbf{s}\in S} \|\omega^*-\omega^*(\cdot-\mathbf{s})\|_{L^2} >0$. 
So if we further reduce $\ep_0$ such that $2C_2(\omega^*)\ep_0 < c(\omega^*)$, \eqref{temp3} would imply that $\mathbf{s}:[0,\infty)\to S$ must stay constant in time. Since we already have $\mathbf{s}(0)=\mathbf{0}$, this gives $\mathbf{s}(t)=\mathbf{0}$ for all $t\geq 0$.

Finally, plugging in  $\mathbf{s}(t)=\mathbf{0}$ for all $t\geq 0$ into \eqref{temp1} finishes the proof of \eqref{goal_dist}. 
\end{proof}

    We move on to the proof of \eqref{goal_deriv}.
    
    \begin{proof}[\textbf{\textup{Proof of \eqref{goal_deriv}}}]
    Since $\mathbf{p}\in C^1(\mathbb{R}^+;\mathbb{T}^2)$ solves \eqref{eq_matrix}, taking the time derivative of both sides gives $\dot{\mathbf{p}} := \frac{d}{dt}\mathbf{p} \in C(\mathbb{R}^+; \mathbb{R}^2)$, and
\[
K \dot{\mathbf{p}}(t) = \dot{\mathbf{b}}(t) = -\frac{d}{dt}\left(
\begin{matrix}\arg\hat \omega(t,\mathbf{k}^{(1)})\\\arg\hat \omega(t,\mathbf{k}^{(2)})\end{matrix}\right) \quad\text{ for all }t\geq 0,
\]
where the time derivative at $t=0$ is understood in the sense of right derivative. Since $K$ is invertible, to prove \eqref{goal_deriv}, it suffices to prove 
\begin{equation}\label{dt_arg}
\Big|\frac{d}{dt} \arg \hat\omega(t,\mathbf{k}^{(j)})\Big| \leq C(\omega^*)\ep \quad\text{ for } j=1,2, ~t\geq 0.
\end{equation}
Without loss of generality, it suffices to prove \eqref{dt_arg} for $j=1$.      
\medskip

    \noindent \textbf{Step 1. Time derivative of argument.} Consider a complex-valued function $z(t) = r(t)e^{i\theta(t)}$ where $r(t)\in C^1(\mathbb{R}^+;\mathbb{R}^+)$ and $\theta(t)\in C^1(\mathbb{R}^+;\mathbb{T})$. If $z(t)\neq 0$ for all $t\geq 0$, taking the derivative in $t$ on both sides of $z=re^{i\theta}$ and divide by $z(t)$, we have
    \[
    \frac{z'(t)}{z(t)} = \frac{r'(t) e^{i\theta(t)}}{z(t)} + i \theta'(t) = \frac{r'(t)}{r(t)} + i\theta'(t),
    \]
hence    \[
    \frac{d}{dt}\arg z(t) = \theta'(t) = \Im \frac{z'(t)}{z(t)}, 
    \]
    where $\Im w$ is the imaginary part of a complex number $w$. Applying this identity to $z(t) = \hat\omega(t,\mathbf{k}^{(1)})$ gives
    \[
  \frac{d}{dt} \arg \hat\omega(t,\mathbf{k}^{(1)}) = \Im \frac{\frac{d}{dt} \hat\omega(t,\mathbf{k}^{(1)}) }{\hat\omega(t,\mathbf{k}^{(1)})}.
    \]
Note that the denominator $\hat\omega(t,\mathbf{k}^{(1)})$ is indeed non-zero for all $t\geq 0$: Due to \eqref{parseval} and the fact that $\hat \omega^*(\mathbf{k}^{(1)})\neq 0$, if $\ep_0(\omega^*)$ is chosen sufficiently small, we can ensure $|\hat\omega(t,\mathbf{k}^{(1)})| \geq \frac{1}{2} |\hat\omega^*(\mathbf{k}^{(1)})| > 0$  for all $t\geq 0$. Therefore, to show \eqref{dt_arg}, it suffices to show
\begin{equation}
\label{dt_hat}
\Big|\frac{d}{dt} \hat\omega(t,\mathbf{k}^{(1)})\Big| \leq C(\omega^*)\ep \quad\text{ for all }t\geq 0.
\end{equation}
Let us compute the time derivate on the left hand side. Let $\omega(t,\cdot)$ to be \emph{any} $C^1$-in-time solution of \eqref{eq vorticity form: Euler eq} (with any initial data). Integration by parts yields
	\begin{align}
\frac{d}{dt} \hat\omega(t,\mathbf{k}^{(1)})
		& = \frac{d}{dt}\int_{\TT^2} e^{-i \mathbf{k}^{(1)} \cdot \mathbf{x}} \omega(t,\mathbf{x}) d\mathbf{x} 
		= \int_{\TT^2} \na (e^{-i \mathbf{k}^{(1)} \cdot \mathbf{x}}) \cdot \mathbf{u}(t,\mathbf{x}) \omega(t,\mathbf{x}) d\mathbf{x} \notag\\
		& =-i \int_{\TT^2}  e^{-i \mathbf{k}^{(1)} \cdot \mathbf{x}} \,\mathbf{k}^{(1)}\cdot \mathbf{u}(t,\mathbf{x}) \omega(t,\mathbf{x})d\mathbf{x} .
		\label{time derivative key term}
	\end{align}
At the first glance, it is unclear whether the integral on the right hand side is small: note that even for a steady state, its corresponding velocity field $\mathbf{u}$ is non-zero. But since $\omega(t,\cdot)$ is close to $\omega^*(\cdot-\mathbf{a}(t))$ in the sense of \eqref{orbital stability},  we will make the following key observation to exploit some cancellation as we replace $\omega(t,\cdot)$ by a steady state.

\medskip
    
    \noindent \textbf{Step 2. Key cancellation using $\omega^*$.}
	Since $\omega^*(t,\cdot) := \omega^*$ is a (steady) solution to \eqref{eq vorticity form: Euler eq}, applying \eqref{time derivative key term} to $\omega^*(t,\cdot)$ yields
	\[
	\int_{\TT^2}  e^{-i \mathbf{k}^{(1)} \cdot \mathbf{x}} \,\mathbf{k}^{(1)}\cdot \mathbf{u}^*(\mathbf{x}) \omega^*(\mathbf{x})d\mathbf{x} =0,
	\]
	and the substitution $\mathbf{x}=\mathbf{y}-\mathbf{a}(t)$ gives
	\begin{equation}\label{temp4}
	\int_{\TT^2}  e^{-i \mathbf{k}^{(1)} \cdot \mathbf{y}} \,\mathbf{k}^{(1)}\cdot \mathbf{u}^*(\mathbf{y}-\mathbf{a}(t)) \omega^*(\mathbf{y}-\mathbf{a}(t))d\mathbf{y} =0,
	\end{equation}
	where we dropped the constant-in-$\mathbf{y}$ factor $e^{i\mathbf{k}^{(1)}\cdot \mathbf{a}(t)}$ in the integral.
	Multiplying $-i$ to \eqref{temp4} and subtract it from \eqref{time derivative key term}, we have
	\begin{equation}
	\begin{split}
	\label{hat_der2}
	\frac{d}{dt} \hat\omega(t,\mathbf{k}^{(1)}) = -i \int_{\TT^2}  e^{-i \mathbf{k}^{(1)} \cdot \mathbf{x}} &\,\mathbf{k}^{(1)}\cdot \Big[\mathbf{u}(t,\mathbf{x}) \big(\omega(t,\mathbf{x})-\omega^*(\mathbf{x}-\mathbf{a}(t)\big)\\
	&+ \big(\mathbf{u}(t,\mathbf{x})-\mathbf{u}^*(\mathbf{x}-\mathbf{a}(t))\big) \omega^*(\mathbf{x}-\mathbf{a}(t)) \Big] d\mathbf{x}.
	\end{split}
	\end{equation}
	    
    \noindent \textbf{Step 3. Completing the estimate \eqref{dt_hat}.} From \eqref{hat_der2}, we immediately have
    \[
    \Big|\frac{d}{dt} \hat\omega(t,\mathbf{k}^{(1)})\Big| \leq |\mathbf{k}^{(1)}| \Big( \|\mathbf{u}(t,\cdot)\|_{L^2} \|\omega(t,\cdot)-\omega^*(\cdot-\mathbf{a}(t))\|_{L^2} + \|\mathbf{u}(t,\cdot)-\mathbf{u}^*(\cdot-\mathbf{a}(t))\|_{L^2} \|\omega^*\|_{L^2} \Big),
    \]
    Combining \eqref{eq_stab}, the $L^2$ boundedness of Biot--Savart operator $\na^\perp \Delta_{\TT^2}^{-1}$ on torus $\TT^2$, and the conservation of $L^p$ norm of vorticity under 2D Euler flow, there exists $ C(\omega^*)>0$ such that
\[
\Big|\frac{d}{dt} \hat\omega(t,\mathbf{k}^{(1)})\Big| \leq C(\omega^*) \|\omega(t,\cdot)-\omega^*(\cdot-\mathbf{a}(t))\|_{L^2}\leq  C(\omega^*) \ep,
\]
    which gives \eqref{dt_hat}. This finishes the proof of \eqref{goal_deriv}.
\end{proof}

\subsection{Evolution in the moving frame centered at $\mathbf{p}(t)$}\label{subsection_moving_frame}
Under the assumptions of Proposition~\ref{prop_p}, for $\omega_0\in C_0^1(\mathbb{T}^2)$ that satisfies $\|\omega_0\|_{L^\infty}\leq \|\omega^*\|_{L^\infty}+3$ and is sufficiently close to $\omega^*$ in $L^2$, let $\mathbf{p} \in C^1(\mathbb{R}^+, \mathbb{T}^2)$ be given by Proposition~\ref{prop_p}, and we define
\begin{equation}\label{def_f}
\rho(t,\mathbf{x}) := \omega(t, \mathbf{x}+\mathbf{p}(t)) \quad\text{ for } t\in \mathbb{R}^+, \mathbf{x}\in\mathbb{T}^2.
\end{equation}
In other words, $\rho$ is the vorticity in a moving frame centered at $\mathbf{p}(t)$ at time $t$. 

 Since $\omega$ satisfies \eqref{eq vorticity form: Euler eq}, one can easily check that $\rho$ satisfies a transport equation
\begin{equation}\label{eq_trans}
\partial_t \rho + \mathbf{v}\cdot \nabla \rho=0,
\end{equation}
  where the velocity field $\mathbf{v}(t,\cdot)$ is divergence free, and it is the Biot--Savart law applied to $\rho$ plus an artificial drift velocity $-\dot{\mathbf{p}}(t)$ resulted from the moving frame. That is,
 \begin{equation}
 \label{def_v}
 \mathbf{v}(t,\cdot) = \nabla^{\perp}\Delta^{-1} \rho(t,\cdot) - \dot{\mathbf{p}}(t).
 \end{equation}

 By \eqref{goal_dist} in Proposition~\ref{prop_p}, we have $
\|\rho(t,\cdot) - \omega^*\|_{L^2(\mathbb{T}^2)}
$
remains small for all $t\geq 0$. Below, we prove a simple lemma showing that this implies smallness of $\|\mathbf{v}(t,\cdot)-\mathbf{u}^*\|_{L^\infty}$ -- relying crucially on smallness of $|\dot{\mathbf{p}}(t)|$ from \eqref{goal_deriv}.

\begin{lemma}\label{prop_v}Under the assumptions of Proposition~\ref{prop_p}, for any $\ep\in (0,\ep_0)$, there exists some $\delta_0=\delta_0(\ep,\omega^*)>0$ (here both $\ep_0$ and $\delta_0$ are given by Proposition~\ref{prop_p}), such that for all $\omega_0\in C_0^1(\mathbb{T}^2)$ satisfying $\|\omega_0-\omega^*\|_{L^2(\mathbb{T}^2)} \leq \delta_0$, the following holds.

Let $\rho$ be given by \eqref{def_f}, with $\mathbf{p}(t)$  given by Proposition~\ref{prop_p}. Then $\rho$ satisfies the transport equation \eqref{eq_trans}, and the velocity field $\mathbf{v}$ in \eqref{def_v} satisfies the estimate
\begin{equation}\label{goal_v}
\|\mathbf{v}(t,\cdot)-\mathbf{u}^*\|_{L^\infty(\mathbb{T}^2)}\leq C_1(\omega^*)\sqrt{\ep} \quad\text{ for all }t\geq 0,
\end{equation}
where $C_1(\omega^*)$ is a constant only depending on $\omega^*$.
\end{lemma}

\begin{proof}
Recall that $\mathbf{v} = \nabla^{\perp}\Delta^{-1} \rho(t,\cdot) - \dot{\mathbf{p}}(t)$, where we already have $|\dot{\mathbf{p}}|\leq C(\omega^*)\ep$ by \eqref{goal_deriv}.
Therefore, to show \eqref{goal_v}, it suffices to prove
\begin{equation}\label{goal_u_diff}
\|\nabla^{\perp}\Delta^{-1} (\rho(t,\cdot) -\omega^*)\|_{L^\infty} \leq C(\omega^*)\sqrt{\ep},
\end{equation}
where we used the Biot--Savart law $\mathbf{u}^*=\nabla^{\perp} \Delta^{-1}\omega^*$.

 By \eqref{goal_dist} and the definition of $\rho$, we have 
$ \|\rho(t,\cdot)-\omega^*\|_{L^2}\leq C(\omega^*)\ep$   for all $t\geq 0$.
 Combining this with the fact that $\|\rho(t,\cdot)\|_{L^\infty}=\|\omega_0\|_{L^\infty}\leq \|\omega^*\|_{L^\infty}+3$, the Cauchy--Schwarz inequality gives
  \begin{equation}\label{temp_l4}
 \|\rho(t,\cdot)-\omega^*\|_{L^4}\leq C(\omega^*)\sqrt{\ep} \quad\text{ for all }t\geq 0.
 \end{equation}

Next, similarly to \cite{Z}, we can compute the Biot--Savart law on $\TT^2$ by extending $\rho(t,\cdot)-\omega^*$ periodically to $\mathbb{R}^2$ and applying the Biot--Savart law in $\mathbb{R}^2$. This gives
\begin{align*}
	 \na^\perp \Delta_{\mathbb{T}^2}^{-1}(\rho(t, \mathbf{x}) -  \omega^*(\mathbf{x})) 
	 & = \frac{1}{2 \pi} \sum_{\mathbf{m} \in \ZZ^2}
 \int_{[-\pi, \pi]^2} \frac{(\mathbf{x}  -\mathbf{y} + 2\pi \mathbf{m})^\perp}{|\mathbf{x} -\mathbf{y} + 2\pi \mathbf{m}|^2} \left( \rho(t,\mathbf{y}) - \omega^*(\mathbf{y}) \right) dy.
\end{align*}
Let us first control the contribution of the terms with $|\mathbf{m}| \ge 2$. Using the bound
\begin{align*}
& \quad \Bigg| \frac{(\mathbf{x} -\mathbf{y} + 2\pi \mathbf{m})^\perp}{|\mathbf{x} -\mathbf{y} + 2\pi \mathbf{m} |^2}   - \frac{(2 \pi \mathbf{m})^\perp}{|2 \pi \mathbf{m}|^2} \Bigg|
\le \frac{100}{|\mathbf{m}|^2} \quad \text{for all} \; |\mathbf{m}| \ge 2 \; \text{ and } \; \mathbf{x} , \mathbf{y} \in [-\pi, \pi]^2,
\end{align*} 
we have the following for all $\mathbf{x} \in [-\pi, \pi]^2$, where $C>0$ is a universal constant:
\begin{align*}
 & \quad \sum_{\mathbf{m} \in \mathbb{Z}^2, |\mathbf m | \ge 2} \Bigg| \int_{[-\pi, \pi]^2} \frac{(\mathbf{x} -\mathbf{y} + 2\pi \mathbf{m})^\perp}{|\mathbf{x} -\mathbf{y} + 2\pi \mathbf{m} |^2}  \left( \rho(t,\mathbf{y}) - \omega^*(\mathbf{y}) \right) dy \Bigg|   \\
 & = \sum_{\mathbf{m} \in \mathbb{Z}^2, |\mathbf m | \ge 2} \Bigg| \int_{[-\pi, \pi]^2}   \Bigg( \frac{(\mathbf{x} -\mathbf{y} + 2\pi \mathbf{m})^\perp}{|\mathbf{x} -\mathbf{y} + 2\pi \mathbf{m} |^2} - \frac{(2 \pi \mathbf{m})^\perp}{ |2 \pi \mathbf{m}|^2} \Bigg) \left( \rho(t,\mathbf{y}) - \omega^*(\mathbf{y}) \right) dy \Bigg| \\
 & \le \sum_{\mathbf{m} \in \mathbb{Z}^2, |\mathbf{m}| \ge 2}  \frac{100}{|\mathbf{m}|^2} \| \rho(t, \cdot) - \omega^* \|_{L^1(\mathbb{T}^2)} \le C \| \rho(t,\cdot) - \omega^* \|_{L^4(\TT^2)},
\end{align*}
where the equality is due to $\rho(t, \cdot) - \omega^*$ having mean zero on $[-\pi,\pi]^2$, and the last inequality follows from the H\"older's inequality and the summability of $\sum_{|\mathbf{m}| \ge 2} \frac{1}{\mathbf{m}^2}$.

Finally, we apply the H\"older's inequality again to bound the contribution of terms with $|\mathbf{m}| <2$. For all $\mathbf{x} \in [-\pi, \pi]^2$, we have
\begin{align*}
 & \quad \sum_{\mathbf{m} \in \mathbb{Z}^2, |\mathbf m | < 2} \Bigg| \int_{[-\pi, \pi]^2} \frac{(\mathbf{x} -\mathbf{y} + 2\pi \mathbf{m})^\perp}{|\mathbf{x} -\mathbf{y} + 2\pi \mathbf{m} |^2}  \left( \rho(t,\mathbf{y}) - \omega^*(\mathbf{y}) \right) dy \Bigg|   \\
 & \le \sum_{\mathbf{m} \in \mathbb{Z}^2, |\mathbf{m}| < 2} \Big\| \frac{1}{|\mathbf{x} - \cdot + 2 \pi \mathbf{m}|} \Big\|_{L^\frac{4}{3}\left([-\pi,\pi]^2\right)}  \| \rho(t,\cdot) - \omega^* \|_{L^4(\TT^2)} \\
 & \le C \| \rho(t, \cdot) - \omega^* \|_{L^4(\TT^2)},
 \end{align*}
 where $C>0$ is a universal constant.
Combining the above estimates with \eqref{temp_l4} yields \eqref{goal_u_diff}, which finishes the proof of \eqref{goal_v}.
\end{proof}
	
	In the next lemma, we take a quick detour to discuss the velocity field $\mathbf{u}^*$ near the saddle point $\mathbf{x}_0$. Using Definition \ref{def: saddle point of w*}, we can find a small parallelogram $D$ centered at $\mathbf{x}_0$, such that the flux of $\mathbf{u}^*$ is strictly inwards along one pair of parallel sides, and strictly outwards along the other pair. See Figure~\ref{fig_parallel} for an illustration.
    
    \begin{figure}[htbp]
    \begin{center}
    \includegraphics[scale=0.8]{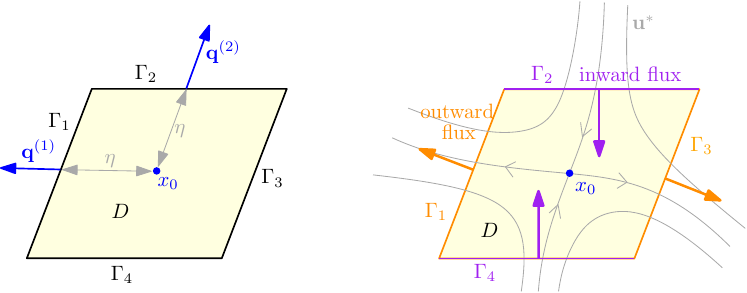}  \\
   (a) \hspace{5cm} (b) \hspace{1cm}
   
    \caption{\label{fig_parallel} (a) Illustration of the parallelogram domain $D$ defined in \eqref{def_D}. (b) Illustration of the streamline of $\mathbf{u}^*$ near the saddle point $\mathbf{x}_0$, and the inward/outward flux along the four sides.}
    \end{center}
        \end{figure}
    
    \begin{lemma}
    \label{lemma_u}Assume $\mathbf{x}_0 \in \TT^2$ is a saddle point of the flow $\mathbf{u}^*$ of a steady state $\omega^*$ in the sense of Definition \ref{def: saddle point of w*}.
   Then there exist two small constants $c_0(\omega^*), \eta(\omega^*)>0$, two different unit vectors $\mathbf{q}^{(1)}, \mathbf{q}^{(2)}$, such that for the open parallelogram $D\subset\mathbb{T}^2$ centered at $\mathbf{x}_0$ given by
   \begin{equation}\label{def_D}
    D := \{ \mathbf{x}=\mathbf{x}_0 + b_1 \mathbf{q}^{(1)} + b_2 \mathbf{q}^{(2)}: |b_1|<\eta, |b_2|<\eta\},
    \end{equation}
   the velocity flux on its four sides  $\Gamma_1,\dots,\Gamma_4$\footnote{Here the four sides are named in clockwise order; see Figure~\ref{fig_parallel}(a) for an illustration. Also, we define each of them as an open line segment, where the vertices are not included.} satisfies 
\begin{equation}\label{goal_flux}
\mathbf{u}^*(t,\mathbf{x})\cdot \mathbf{n} > c_0(\omega^*) ~~\text{ on }\Gamma_1\cup\Gamma_3 \quad \text{ and } \quad \mathbf{u}^*(t,\mathbf{x})\cdot \mathbf{n} < -c_0(\omega^*) ~~\text{ on }\Gamma_2\cup\Gamma_4,
\end{equation}
where $\mathbf{n}$ is the outer normal of $D$.
    \end{lemma}
    
    \begin{proof}
    
        For any $\mathbf{x}\in\mathbb{T}^2$ with $|\mathbf{x}-\mathbf{x}_0|<\frac{1}{2}$, we can express $\mathbf{x}-\mathbf{x}_0$ as a vector in $\mathbb{R}^2$ (with length less than $\frac12$), and estimate $\mathbf{u}^*(\mathbf{x})$ as follows (where we used the fact that $\mathbf{u}^*(\mathbf{x}_0)=\nabla^\perp \Phi^*(\mathbf{x}_0) = \mathbf{0}$):
    \begin{equation}
    \label{estimate_u}
    \mathbf{u}^*(\mathbf{x}) = \begin{pmatrix} -\partial_2\Phi^*(\mathbf{x}) \\ \partial_1 \Phi^*(\mathbf{x})\end{pmatrix} = \underbrace{\begin{pmatrix}
    -\partial_{12}\Phi^*(\mathbf{x}_0) & -\partial_{22} \Phi^*(\mathbf{x}_0) \\ \partial_{11}\Phi^*(\mathbf{x}_0) & \partial_{12}\Phi^*(\mathbf{x}_0) \end{pmatrix}}_{=: A} (\mathbf{x}-\mathbf{x_0}) +  \mathbf{e}(\mathbf{x}-\mathbf{x}_0),
    \end{equation}
    where the error term $\mathbf{e}(\mathbf{x}-\mathbf{x}_0)$ satisfies $|\mathbf{e}(\mathbf{x}-\mathbf{x}_0)| \leq C(\omega^*) |\mathbf{x}-\mathbf{x}_0|^2$.
    
    It can be easily checked that the matrix $A$ has eigenvalues
    \[
    \lambda=\pm \sqrt{(\partial_{12}\Phi^*(\mathbf{x}_0))^2 - \partial_{11}\Phi^*(\mathbf{x}_0) \partial_{22}\Phi^*(\mathbf{x}_0)  },
    \] where the expression in the square root is strictly positive since it is $-\det(D^2\Phi^*(\mathbf{x}_0))$: recall that the definition of a saddle point requires $D^2\Phi^*(\mathbf{x}_0)$ having two eigenvalues of opposite signs, indicating that it has a strictly negative determinant.
    
     Let us denote the two eigenvalues of $A$ by $\lambda_1>0$ and $\lambda_2<0$, and denote by $\mathbf{q}^{(1)}$ and $\mathbf{q}^{(2)}$ their corresponding eigenvectors (each is a column vector of unit length). With such notations, we can decompose the matrix $A$ as
    \[
   A = Q \Lambda Q^{-1}, \quad \text{ where }  Q := \begin{pmatrix}\mathbf{q}^{(1)} & \mathbf{q}^{(2)}\end{pmatrix}  \text{ and } \Lambda := \begin{pmatrix} \lambda_1 & 0 \\ 0  & \lambda_2\end{pmatrix}.
   \]
    
   Let us consider the parallelogram domain $D\subset\mathbb{T}^2$ given by \eqref{def_D},   where $\eta\in (0,1/4)$ is a small parameter to be determined later.
     
    On the side $\Gamma_1 =  \{ \mathbf{x}=\mathbf{x}_0 + \eta \mathbf{q}^{(1)} + b_2 \mathbf{q}^{(2)}:  |b_2|<\eta\}$, we have
    \[
    \mathbf{u}^*(\mathbf{x}) = A (\eta \mathbf{q}^{(1)} + b_2 \mathbf{q}^{(2)}) + \mathbf{e}(\mathbf{x}-\mathbf{x}_0) = \lambda_1 \eta \mathbf{q}^{(1)}+ \lambda_2 b \mathbf{q}^{(2)} + O(\eta^2),
    \]
    so along $\Gamma_1$, using the fact that $\mathbf{n}\cdot \mathbf{q}^{(1)}>0$ and $\mathbf{n}\cdot \mathbf{q}^{(2)}=0$, we have
    \[
    \mathbf{n}\cdot  \mathbf{u}^*(\mathbf{x}) = \lambda_1\eta \mathbf{n}\cdot \mathbf{q}^{(1)} + O(\eta^2) =: \lambda_1 \eta \sin(\theta^*)+ O(\eta^2),
    \]
    where $\theta^*\in (0,\pi)$ is the angle between $\mathbf{q}^{(1)}$ and $\mathbf{q}^{(2)}$. Therefore by setting $\eta>0$ sufficiently small, we can ensure that 
    \[\mathbf{n}\cdot  \mathbf{u}^*(\mathbf{x})>\frac{1}{2}\lambda_1 \eta \sin(\theta^*)\quad\text{ on }\Gamma_1.
    \] We can argue likewise for the other three sides of $\partial D$ to conclude \eqref{goal_flux}, where $c_0(\omega^*)$ is given by 
   $ \frac{1}{2} \eta \sin(\theta^*) \min\{\lambda_1, |\lambda_2|\}.
    $
    \end{proof}
    
    Even though $\mathbf{v}(t,\cdot)$ is close to $\mathbf{u}^*$ in $L^\infty$ by Lemma~\ref{prop_v}, it may not have a non-degenerate saddle point in $D$. Nevertheless, combining Lemma~\ref{prop_v} and  Lemma~\ref{lemma_u},  we easily arrive at the conclusion that on the sides of $\partial D$, $\mathbf{v}(t,\cdot)$ still has strictly outward flux on $\Gamma_1, \Gamma_3$, and strictly inward flux on $\Gamma_2, \Gamma_4$.
    
    \begin{corollary}\label{cor_v}
    Under the assumptions of Proposition~\ref{prop_p},  there exists some $\delta_1=\delta_1(\omega^*)>0$, such that for all $\omega_0\in C_0^1(\mathbb{T}^2)$ satisfying $\|\omega_0-\omega^*\|_{L^2(\mathbb{T}^2)} \leq \delta_1$, the following holds for all $t\geq 0$:

Let $\rho$ and $\mathbf{v}$ be given by \eqref{def_f} and \eqref{def_v} respectively, with $\mathbf{p}(t)$  given by Proposition~\ref{prop_p}. Then the following holds:
\begin{enumerate}[label=\textup{(\alph*)}] \item On the sides $\Gamma_1,\cdots\Gamma_4$ of the parallelogram $D$ in Lemma~\ref{lemma_u}, we have
\begin{equation}\label{goal_v_flux}
\mathbf{v}(t,\mathbf{x})\cdot \mathbf{n} > \frac{1}{2}c_0(\omega^*) ~~\text{ on }\Gamma_1\cup\Gamma_3 \quad \text{ and } \quad \mathbf{v}(t,\mathbf{x})\cdot \mathbf{n} < -\frac{1}{2}c_0(\omega^*) ~~\text{ on }\Gamma_2\cup\Gamma_4,
\end{equation}
where $\mathbf{n}$ is the outer normal of $D$, and $c_0(\omega^*)$ is as in Lemma~\ref{lemma_u}.

\smallskip
\item $|\mathbf{p}(0)| \leq \frac{\eta}{10}$, where $\eta=\eta(\omega^*)$ is given in Lemma~\ref{lemma_u}.
\end{enumerate}
    \end{corollary}

  \begin{proof}  To prove (a),  let us define $\ep_1 := \min\{\frac{1}{2}\ep_0(\omega^*), (\frac{1}{2}c_0(\omega^*) C_1(\omega^*)^{-1})^2\}$, 
    where $\ep_0(\omega^*)$, $C_1(\omega^*)$ and $c_0(\omega^*)$ are given in Proposition~\ref{prop_p},  \eqref{goal_v} and Lemma~\ref{lemma_u} respectively. 
    Applying Lemma~\ref{prop_v} to $\ep=\ep_1$, we know that there exists some $\delta_0(\ep_1,\omega^*)$ which essentially depends on $\omega^*$ only (since $\ep_1$ depends on $\omega^*$), such that for all $\omega_0\in C_0^1(\mathbb{T}^2)$ satisfying $\|\omega_0-\omega^*\|_{L^2} \leq \delta_0$, \eqref{goal_v} holds with $\ep$ replaced by $\ep_1$, i.e.
    \[
    \|\mathbf{v}(t,\cdot)-\mathbf{u}^*\|_{L^\infty(\mathbb{T}^2)} \leq C_1(\omega^*) \sqrt{\ep_1} \leq \frac{1}{2}c_0(\omega^*).
    \]
    Combining this with \eqref{goal_flux} in Lemma~\ref{lemma_u}  immediately  leads to \eqref{goal_v_flux}.
    
    \medskip
    
    We now move on to (b). First, we point out that \eqref{omega_trans} implies
    \[
    c(\omega^*) := \min_{\mathbf{s}\in\mathbb{T}^2, |\mathbf{s}|\geq\frac{\eta}{10}} \|\omega^* - \omega^*(\cdot-\mathbf{s})\|_{L^2(\mathbb{T}^2)} > 0.
    \]
    This is because the function $\mathbf{s}\mapsto \|\omega^* - \omega^*(\cdot-\mathbf{s})\|_{L^2(\mathbb{T}^2)}$ is continuous in $\mathbb{T}^2$, therefore its minimum in the closed set $\{s\in \mathbb{T}^2: |\mathbf{s}|\geq \frac{\eta}{10}\}$ is achieved at some point. If the minimum value is 0, it would contradict \eqref{omega_trans}.

    By Proposition~\ref{prop_p} (in particular \eqref{goal_dist} at $t=0$), there exists some $\delta_0(\omega^*)$, such that for all $\omega_0\in C_0^1(\mathbb{T}^2)$ satisfying $\|\omega_0-\omega^*\|_{L^2} \leq \delta_0$, we have
    \begin{equation}\label{temp4}
    \|\omega_0-\omega^*(\cdot-\mathbf{p}(0))\|_{L^2} \leq \frac{1}{4} c(\omega^*).
    \end{equation}
    So if $\delta_1$ is chosen such that $\delta_1\leq\frac{1}{4} c(\omega^*)$, we can combine the assumption $\|\omega_0-\omega^*\|_{L^2}\leq \delta_1\leq\frac{1}{4} c(\omega^*)$ with \eqref{temp4} to obtain
    \[
      \|\omega^*-\omega^*(\cdot-\mathbf{p}(0))\|_{L^2} \leq \frac{1}{2} c(\omega^*) < \min_{\mathbf{s}\in\mathbb{T}^2, |\mathbf{s}|\geq\frac{\eta}{10}} \|\omega^* - \omega^*(\cdot-\mathbf{s})\|_{L^2(\mathbb{T}^2)},
    \]
    which implies that $|\mathbf{p}(0)| \leq \frac{\eta}{10}$.
    
    Finally, we set $\delta_1(\omega^*) := \min\{\delta_0(\ep_1,\omega^*), \frac{1}{4}c(\omega^*)\}$, so  both (a) and (b) hold when $\|\omega_0-\omega^*\|_{L^2}\leq \delta_1$.
         \end{proof}

\subsection{Proof of superlinear growth in the torus}
\label{sec_torus_growth}
We are now ready to prove Theorem~\ref{thm: torus}.

\begin{proof}[\textbf{\textup{Proof of Theorem~\ref{thm: torus}}}]
Given $\omega^*$ and $\ep>0$, we construct $\tilde\omega_0 \in C_0^\infty(\mathbb{T}^2)$ such that the following holds, where $\delta_1(\omega^*)$ is given by Corollary~\ref{cor_v}:

\begin{enumerate}[label=(\alph*)]
\item
 $\|\tilde\omega_0 - \omega^*\|_{L^2(\mathbb{T}^2)} < \min\{\ep, \frac{1}{2}\delta_1(\omega^*)\}$.

\item $\tilde\omega_0 > 2$ on $L_1:= \{\mathbf{x}=\mathbf{x}_0+b_1\mathbf{q}^{(1)}+\frac{\eta}{2}\mathbf{q}^{(2)}: |b_1|\leq 2\eta\}$,
and\\
  $\tilde\omega_0 < -2$ on $L_2 := \{\mathbf{x}=\mathbf{x}_0+b_1\mathbf{q}^{(1)}-\frac{\eta}{2}\mathbf{q}^{(2)}: |b_1|\leq 2\eta\}$. 
  
  \noindent
  Here the vectors $\mathbf{q}^{(1)}, \mathbf{q}^{(2)}$ and the constant $\eta$ are given by Lemma~\ref{lemma_u}. Note that the segments $L_1, L_2$ are both parallel to the sides $\Gamma_2$ and $\Gamma_4$. Each of them have a portion inside $D$, and their endpoints both extends outside $\partial D$ by distance $\eta$.

\item $\| \tilde\omega_0 \|_{L^\infty(\TT^2)} \le  \| \omega^* \|_{L^\infty(\TT^2)} + 3.$

\end{enumerate}
\medskip

It is easy to see the existence of such $\tilde\omega_0$: one can first mollify $\omega_0$ to make it $C^\infty$ (note that the $L^2$ difference can be made arbitrarily small after the mollification), then modify it in two sufficiently small neighborhoods containing the curves $L_1$ and $L_2$ such that (b)--(c) are satisfied, and meanwhile ensure that the mean-zero condition is still satisfied after the modification.

For such $\tilde\omega_0$, let us consider any $\omega_0 \in C^1_0(\mathbb{T}^2)$ satisfying 
\[
\|\omega_0 - \tilde\omega_0\|_{L^\infty(\mathbb{T}^2)} \leq \min\Big\{1,  \frac{1}{4\pi} \delta_1(\omega^*)\Big\}.
\]
For such $\omega_0$, we have $\|\omega_0 - \tilde\omega_0\|_{L^2(\mathbb{T}^2)} \leq \frac{1}{2}\delta_1(\omega^*)$, therefore combining it with (a) yields
\[
\|\omega_0 - \omega^*\|_{L^2(\mathbb{T}^2)} \leq \delta_1(\omega^*).
\]
We can then apply Corollary~\ref{cor_v}(a) to conclude that \eqref{goal_v_flux} holds for the velocity field $\mathbf{v}(t,\cdot)$ in the moving frame.

As a result, for the transport equation \eqref{eq_trans} in the moving frame (with $\mathbf{v}$ given by \eqref{def_v}), we have verified that $\mathbf{v}(t,\cdot)$ satisfies Condition 1 in Appendix~\ref{appendix A} for all $t\geq 0$.

Next we will verify that $\rho_0=\rho(0,\cdot) = \omega_0(\cdot+\mathbf{p}(0))$ satisfies Condition 2 in Appendix~\ref{appendix A}. By Property (b) of $\tilde\omega_0$ and the assumption that $\|\omega_0-\tilde\omega_0\|_{L^\infty(\mathbb{T}^2)}\leq 1$, we have $ \omega_0 > 1$ on $L_1$ and $ \omega_0 < -1$ on $L_2$. Since $\rho_0=\omega_0(\cdot+\mathbf{p}(0))$, the above implies 
\[\rho_0 >1 \text{ on }\tilde{L}_1\quad\text{ and }\rho_0 <-1\text{ on }\tilde{L}_2,
\]
 where $\tilde{L}_i$ is the translation of $L_i$ by vector $-\mathbf{p}(0)$ for $i=1,2$. By Corollary~\ref{cor_v}(b), $|\mathbf{p}(0)|\leq \frac{\eta}{10}$, thus each of the two line segments $\mathcal{C}_i := \tilde{L}_i \cap \overline{D}$ has its endpoints on $\Gamma_1$ and $\Gamma_3$, with the interior inside $D$. This shows that $\rho_0$ satisfies Condition 2 in Appendix~\ref{appendix A} on the curves $\mathcal{C}_1$, $\mathcal{C}_2$ defined above.

Finally, since $\mathbf{v}$ and $\rho_0$ satisfy Conditions 1--2, we can apply the superlinear gradient growth result in Lemma~\ref{grad_growth} to the transport equation \eqref{eq_trans} to conclude 
\[
\int_0^\infty \|\nabla\rho(t,\cdot)\|_{L^\infty(\mathbb{T}^2)}^{-1} dt < C(\omega^*).
\]
This finishes the proof since $\|\nabla\omega(t,\cdot)\|_{L^\infty(\mathbb{T}^2)}=\|\nabla\rho(t,\cdot)\|_{L^\infty(\mathbb{T}^2)}$ for all $t\geq 0$.
\end{proof}

\section{Superlinear growth on the whole plane}
\label{sec_R}
In this section, we give an example of superlinear growth of $\|\nabla\omega\|_{L^\infty}$ on the whole plane $\RR^2$. As we discussed in the introduction, our initial data will be chosen as a small perturbation to the Lamb dipole. Below we collect some classical properties of Lamb dipole; see \cite[Section 2.1--2.2]{CJ-Lamb} for a more detailed review.

\subsection{Preliminaries on the Lamb dipole}
\label{subsec_lamb} The Lamb  dipole $\omega_L \in \calX_{odd,+}$ is defined by
\be
\omega_L(\mathbf x) : = \begin{cases}
     \frac{-2 c_L}{J_0(c_L)} J_1(c_L r) \sin \theta &\text{ for } r \le 1, \\
    0 & \text{ for }r>1.
\end{cases}
\label{Lamb dipole: expression}
\ee
Here $J_m(r)$ is the $m$-th order Bessel function of first kind, the constant $c_L >0$ is the first positive zero point of $J_1$, and $J_0(c_L)<0$. From \eqref{Lamb dipole: expression}, one can easily check that $\omega_L$ is compactly supported in $B(0,1)$, and is Lipschitz continuous in $\RR^2$. 

It is known that $\omega_L(\mathbf x -t \mathbf e_1)$ is a traveling wave of the 2D Euler equation \eqref{eq vorticity form: Euler eq} with speed 1 towards the right. Motivated by this fact, we consider the moving frame that moves towards the right with speed 1, i.e.
\begin{equation}
\label{co-moving frame: RR2}
\bar\omega(t,\cdot) := \omega(t,\cdot+t\mathbf{e}_1).
\end{equation}
 If $\omega$ is a solution to \eqref{eq vorticity form: Euler eq} in the original frame, then in the moving frame \eqref{co-moving frame: RR2}, $\bar\omega$ satisfies
\be
\partial_t \bar \omega + (\bar {\mathbf{u}} - \mathbf e_1) \cdot \na \bar \omega=0,
\label{Euler eq: moving frame 1}
\ee
with $\bar{\mathbf{u}} :=\nabla^\perp \Delta^{-1}\bar\omega$.

In particular, $\omega_L$ is a steady state to \eqref{Euler eq: moving frame 1}. Its stream function $\bar{\psi}_L := \Delta^{-1}\omega_L + x_2$ in the moving frame \eqref{co-moving frame: RR2} can be expressed as follows (in polar coordinates):
\be
\bar\psi_L(r,\theta)
= \begin{cases}
    - \frac{2 J_1(c_L r)}{c_L J_0(c_L)} \sin \theta, & r <1, \\
    \left( \frac{1}{r} -r \right) \sin \theta, & r \ge 1.
\end{cases}
\label{stream function: lamb dipole}
\ee

Using the explicit expression  \eqref{stream function: lamb dipole} of the stream function, a direct computation yields that in the moving frame, the Lamb dipole induces two saddle points $(\pm 1,0)$.
 
\begin{lemma}[Saddle points of the Lamb dipole]
\label{lemma: saddle pts of Lamb dipole}
    In the moving frame \eqref{co-moving frame: RR2}, the Lamb dipole has two saddle points at $\mathbf{x}_0^\pm = (\pm 1, 0)$, where the Hessian matrices are
    \[
    D^2\bar\psi_L(\mathbf{x}_0^-) = 
    \left(
    \begin{matrix}
         0 & 2 \\
        2 & 0
    \end{matrix}
    \right)
    \quad \text{ and }\quad  D^2\bar\psi_L(\mathbf{x}_0^+)= 
    \left(
    \begin{matrix}
        0 & -2 \\
        -2 & 0
    \end{matrix}
    \right)
    \text{ respectively}.
    \]
In particular, since the velocity field in the moving frame is given by 
\[
\mathbf{u}_L-\mathbf{e}_1=\nabla^\perp \bar\psi_L,
\] there exist some sufficiently small $\eta \in (0,\frac{1}{10})$ and $c_0>0$, such that the flux of $\mathbf{u}_L-\mathbf{e}_1$  on the four sides\footnote{Let us denote the vertical sides by $\Gamma_1$ and $\Gamma_3$, and horizontal sides by $\Gamma_2$ and $\Gamma_4$. See Figure~\ref{fig_lamb2} for an illustration. The sides do not include the vertices, so the outward normal $\mathbf{n}$ is well-defined.} of the square
\[
Q := \{\mathbf{x}\in\mathbb{R}^2: |x_1+1|< \eta, |x_2|< \eta\}.
\]
satisfies
\begin{equation}\label{u_bdry_ineq}
(\mathbf{u}_L-\mathbf{e}_1)\cdot \mathbf{n} > c_0 \quad\text{ on }\Gamma_1\cup\Gamma_3, \quad\text{ and }(\mathbf{u}_L-\mathbf{e}_1)\cdot \mathbf{n} < -c_0 \quad\text{ on }\Gamma_2\cup\Gamma_4,
\end{equation}
where $\mathbf{n}$ is the outward normal of $\partial Q$.
 See Figure ~\ref{fig_lamb2} for an illustration.
\end{lemma}

	\begin{figure}[htbp]
	\begin{center}
	\includegraphics[scale=1]{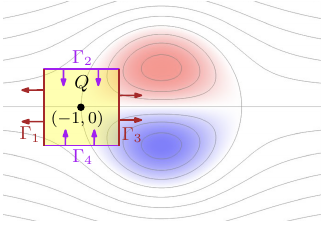}
	\caption{\label{fig_lamb2} In the moving frame \eqref{co-moving frame: RR2}, the velocity field $\mathbf{u}_L - \mathbf{e}_1$ has strictly inward/outward flux along the four sides of a sufficiently small square centered at $(-1,0)$.}
	\end{center}
	\end{figure}

In addition, the Lamb dipole has been verified to be orbitally stable in $\calX_{odd,+}$ by Abe--Choi \cite{lambdipolesta22}. Below we state the orbital stability result in the moving frame \eqref{co-moving frame: RR2}.\footnote{The stability result can either be stated in the original frame or moving frame; the only difference is that the displacement $a(t)$ would be different in those two frames -- they differ by $t$.}
\begin{proposition}[Orbital stability of Lamb dipole, {\cite[Theorem 1.1]{lambdipolesta22}}] 
\label{prop: orbital stability of Lamb dipole}
 In the moving frame \eqref{co-moving frame: RR2}, the Lamb dipole $\omega_L$ is orbitally stable up to a translation, in the sense that for any $\ep >0$ and $M>0$, there exists $\delta>0$, such that for all $\bar \omega_0 \in \calX_{odd,+}$ satisfying
\[
 \| \bar \omega_0 \|_{L^1} \le M  \quad \text{and} \quad 
\| \bar \omega_0  -\omega_L\|_{\calX}  \le \delta,
\]
 the global solution $\bar\omega(t,\cdot)$ of \eqref{Euler eq: moving frame 1} satisfies
\be
\big\| \bar \omega(t,\cdot) - \omega_L(\cdot - a(t) \mathbf e_1) \big\|_{\calX}
< \ep \quad \text{ for all }\; t \ge 0,
\label{orbital stability: lamb-dipole}
\ee
for some $a(t):\mathbb{R}^+ \to \mathbb{R}$ depending on $\bar\omega_0$.
\end{proposition}

Note that if $\bar\omega_0=\omega_L$, clearly \eqref{orbital stability: lamb-dipole} holds for $a(t)\equiv 0$. However, in general it is necessary to allow a translation $a(t)$: for example, if $\bar\omega_0$ is a slightly larger Lamb dipole with speed $1+\ep$ in the original frame, then \eqref{orbital stability: lamb-dipole} holds for $a(t)=\ep t$, and \eqref{orbital stability: lamb-dipole} would fail to hold for $t\gg \ep^{-1}$ if we do not include a displacement $a(t)$. Although the existence of $a(t)$ is shown in \cite{lambdipolesta22}, it was unknown whether $a(t)$ could be chosen to be a regular function of $t$.

\subsection{Constructing an approximate translation $p(t)$}\label{subsec_p}
Similar to  Section \ref{sec_trans_torus},  we aim to introduce an ``approximate translation'' $p\in C^1(\mathbb{R}^+)$ whose speed can be made arbitrarily small if $\bar\omega_0$ is sufficiently close to $\omega_L$, such that \eqref{orbital stability: lamb-dipole} still holds as we replace $a(t)$ by $p(t)$, where we allow the error to be inflated by a constant factor. 

A tempting choice is to set $p(t)$ as the $x_{1}$-coordinate of the center of mass of $\bar\omega(t,\cdot) 1_{\{x_2>0\}}$. This does not work, since a solution initially close to $\omega_L$ may have most  vorticity staying close to some translation of $\omega_L$, while leaving behind a long tail  whose mass stays small but length grows linearly in time. For $t\gg 1$, the center of mass will fail to capture the location of the ``head'' part.

Instead of an explicit definition, we will define $p(t)$ implicitly as follows. For each $t\geq 0$, we define $p(t)$ such that
\be
\int_{\RR_+^2}  \bar\omega(t, \mathbf x ) g(x_1- p(t)) d \mathbf x =0,
\label{R2: center def}
\ee
where $g:\mathbb{R}\to\mathbb{R}$ is a smooth, odd and non-decreasing function satisfying $0\leq g' \leq 3$ and the following:\footnote{To see the motivation for $g$, note that if we had set $g(x_1)=x_1$ for all $x_1\in\mathbb{R}$, \eqref{R2: center def} would lead to $p(t)$ being the horizontal center of mass of $\omega(t,\cdot) 1_{\mathbb{R}^2_+}$. The cut-off outside distance 2 is to ensure that a small mass very far away has little effect on $p(t)$. }
\be
 g(x_1)
:= \begin{cases}
    x_1, & |x_1| \le 2, \\
    3, & x_1 \ge 3, \\
    -3, & x_1 \le 3.
\end{cases}
\label{R2: cutoff function}
\ee

In the next proposition, we establish the existence, uniqueness and regularity properties of $p(t)$ if the initial condition $\bar\omega_0$ is close to $\omega_L$.
\begin{proposition}
\label{lemma: mass center R}Let $m_L := \int_{\RR^2_+}\omega_L d\mathbf{x}$.
For any $\ep \in (0,\frac{1}{100}m_L)$, Proposition~\ref{prop: orbital stability of Lamb dipole} gives the existence of some $\delta_3=\delta_3(\ep)>0$, such that for any  initial data $\bar\omega_0\in \calX_{odd,+}$ satisfying
\be
\label{initial_r2}
|\supp\bar\omega_0| \leq 5, \quad \| \bar \omega_0 \|_{L^\infty} \le \| \omega_L \|_{L^\infty} + 3  \quad \text{and} \quad \| \bar \omega_0  -\omega_L\|_{\calX}  \le \delta_3,
\ee
the  global-in-time solution $\bar\omega$ to \eqref{Euler eq: moving frame 1} with initial data $\bar\omega_0$ satisfies
\be
\| \bar \omega(t,\cdot) - \omega_L(\cdot -a(t) \mathbf e_1) \|_{L^2} < \ep  \quad \text{ for all }\; t \ge 0
\label{assum: RR}
\ee
for some $a(t) \in \RR$. For such solution $\bar\omega$, we define $p(t)$ by \eqref{R2: center def}. Then there exists a universal constant $C_3>0$ such that the following properties hold:

\begin{enumerate}[label=\textup{(\alph*)}]
\item For each $t\geq 0$, there exists a unique $p(t)$ that satisfies \eqref{R2: center def}. Furthermore, $t \mapsto p(t)$ is differentiable.
\item $p(t)$ is a good approximation of the translation, in the sense that  
    \be
\| \bar \omega(t,\cdot) - \omega_L(\cdot -p(t) \mathbf e_1) \|_{L^2} < C_3 \ep \quad\text{ for all } \; t \ge 0,
\label{mass center: R good approximation}
\ee 
and
\be
\label{u_infty_bd2}
\|\nabla^\perp \Delta^{-1} (\bar \omega(t,\cdot) - \omega_L(\cdot -p(t) \mathbf e_1))\|_{L^\infty} < C_3 \ep \quad\text{ for all } \; t \ge 0.
\ee
\item The time derivative of $p(t)$ satisfies 
    \be
    |\dot p(t)| \le C_3 \ep^{1/2} \quad\text{ for all } \; t \ge 0.
    \label{mass center: R 2}
    \ee

\end{enumerate}

 \end{proposition}

\begin{proof}
\noindent 
\textbf{Step 1. Existence and uniqueness of $p(t)$.}
Let us introduce the function $H:\mathbb{R}^+ \times \mathbb{R}\to\mathbb{R}$ given by 
\begin{equation}\label{def_H}
H(t,p) := \int_{\RR_+^2}  \bar\omega(t, \mathbf x ) g(x_1- p) d\mathbf{x}.
\end{equation}
Since $\bar\omega$ is $C^1$ in $t,x$ and $g\in C^\infty(\mathbb{R})$, we have $H\in C^1_t C^\infty_p$. Also, at each fixed $t$, the dominated convergence theorem gives $\lim_{p\to +\infty} H(t,p) = -3 m_0 <0 $ and $\lim_{p\to -\infty} H(t,p) = 3 m_0 >0 $, where $m_0:= \int_{\RR_+^2}  \bar\omega(t, \mathbf{x}) d\mathbf{x} = \int_{\RR_+^2}\omega_0  d\mathbf{x}>0 $. By intermediate value theorem, there exists some $p(t)\in\mathbb{R}$ such that $H(t,p(t))=0$. 

Next we will show uniqueness of $p(t)$ using \eqref{assum: RR}. Without \eqref{assum: RR}, $p(t)$ may not be unique: for example, at some $t$, if $\bar\omega(t,\cdot)$ is the superposition of two identical Lamb dipoles centered at $(\pm10,0)$ respectively, then $H(t,p)=0$ for any $p\in(-7,7)$.

 For a fixed $t\geq 0$, note that $H(t,p)$ is non-increasing in $p$: taking partial derivative in $p$ yields
\be\label{der_H}
\partial_p  H(t,p)  =-\int_{\RR_+^2}   \bar\omega(t, \mathbf x) g'( x_1-p) d \mathbf x \leq 0, 
\ee
where we used the fact that $g'\geq 0$ and $\bar\omega(t,\cdot) \in \calX_{odd,+}$.  

Let us also define
\[
H_L(p) := \int_{\RR_+^2}  \omega_L( \mathbf x ) g(x_1- p) d\mathbf{x} .
\]
Since $g(x_1)=x_1$ for $|x_1|\leq 2$, and $\omega_L$ is even about $x_1=0$ and supported in $B(0,1)$, we can check that 
\begin{equation}\label{HL}
H_L(p) = -m_L p \quad\text{ for } |p|\leq1, 
\end{equation}
where $m_L := \int_{\mathbb{R}_2^+}\omega_L d\mathbf{x}$.

For any $t\geq 0$, by \eqref{assum: RR}, we can decompose $ \bar\omega(t,\cdot)$ into
\[
 \bar\omega(t,\cdot) =: \bar\omega_L(\cdot-a(t) \mathbf e_1) + \bar\omega_\ep(t,\cdot), \quad \text{ where }\|\bar\omega_\ep(t,\cdot)\|_{L^2}\leq \ep.
\]
Using such decomposition, $H(t,p)$ can also be decomposed as
\begin{equation}\label{dHdp}
H(t,p) = H_L(p-a(t)) + H_\ep(t,p),
\end{equation}
where $H_\ep(t,p)$ is defined the same way as \eqref{def_H} with $\bar\omega$ replaced by $\bar\omega_\ep$.

Using the fact that $|\supp\bar\omega_\ep|<|\supp\bar\omega_0|+|\supp\bar\omega_L|\leq 5+\pi < 10$, we have
\begin{equation}\label{tempH}
|H_\ep(t,p)| \leq \|\bar\omega_\ep\|_{L^2} \|1_{\supp\bar\omega_\ep} \|_{L^2} \|g\|_{L^\infty} \leq 30\ep \quad\text{ for all }t\geq 0, p\in\mathbb{R},
\end{equation}
and by writing $\partial_p H_\ep(t,p)$ as in \eqref{der_H} (with $\bar\omega$ replaced by $\bar\omega_\ep$) and estimating the same way as \eqref{tempH} (note that $\|g'\|_{L^\infty}<3$), we also have 
\[|\partial_p H_\ep(t,p)|\leq 30\ep \quad\text{ for all }t\geq 0, p\in\mathbb{R}.
\]
The above smallness estimates on $H_\ep$ and $\partial_p H_\ep$ yields that at $p=a(t)$, we have
\begin{equation}\label{absH}
|H(t,a(t))|  \leq |\underbrace{H_L(0)}_{=0}+H_\ep(t,a(t))| \leq 30\ep,
\end{equation}
and for all $p\in(a(t)-1,a(t)+1)$, using \eqref{HL}, we have
\begin{equation}
\label{derH}
\partial_p H(t,p) = \underbrace{\partial_p H_L(p-a(t))}_{=-m_L}+\partial_p H_\ep(t,p) \leq -m_L+30\ep < -\frac{1}{2}m_L,
\end{equation}
where we used that $\ep\leq\frac{1}{100}m_L$.

Combining \eqref{absH} and \eqref{derH} together, we know that within the interval $(a(t)-1,a(t)+1)$, there is a unique $p(t)$ such that $H(t,p(t))=0$. Note that such $p(t)$ is also unique in $\mathbb{R}$: there cannot be any solution outside of the interval $(a(t)-1,a(t)+1)$, since the above estimates give $H(t,a(t)+1)<0$ and $H(t,a(t)-1)>0$, and recall that $p\mapsto H(t,p)$ is non-increasing in $p$. Finally, using the implicit function theorem, the differentiability of $t \mapsto p(t)$ is a result of \eqref{derH} together with the regularity of $H(t,p)$.

\medskip
\textbf{Step 2. $p(t)$ is a good approximation of translation}.
Using \eqref{absH} and \eqref{derH}, we can apply mean value theorem to the interval between $p(t)$ and $a(t)$ to conclude
\[|p(t)-a(t)| \leq 60 m_L^{-1}\ep.
\]
This directly leads to
$
\|\omega_L(\cdot-p(t)\mathbf{e}_1) - \omega_L(\cdot-a(t)\mathbf{e}_1) \|_{L^2} \leq C\ep
$
for some universal constant $C$, since $\omega_L$ is compactly supported and Lipschitz continuous.
Combining this with \eqref{assum: RR} and applying triangle inequality yields \eqref{mass center: R good approximation}.

To show \eqref{u_infty_bd2}, let us introduce \begin{equation}\label{def_w_p_ep}
 \bar\omega_{p,\ep}(t,\cdot) := \bar\omega(t,\cdot) - \omega_L(\cdot-p(t)\mathbf{e}_1), 
\end{equation}
where we have $ \|\bar\omega_{p,\ep}\|_{L^2}\leq C\ep$ by \eqref{mass center: R good approximation}. 
We also define 
\be\label{def_upe}
\mathbf{u}_L := \nabla^\perp \Delta^{-1} \omega_L\quad\text{ and }\quad\bar{\mathbf{u}}_{p,\ep} := \nabla^\perp \Delta^{-1} \bar{\omega}_{p,\ep}.
\ee
For any $\mathbf{x}\in\RR^2$ we have
\begin{equation}\label{u_ep_bd}
\begin{split}
|\bar{\mathbf{u}}_{p,\ep}(t,\mathbf{x})| & \leq  C \int_{|\mathbf{x}-\mathbf{y}|\leq 1}  \frac{|\bar \omega_{p,\ep}(t,\mathbf y)|}{|\mathbf x- \mathbf y|} d \mathbf y + C  \int_{|\mathbf{x}-\mathbf{y}|> 1}  \frac{|\bar \omega_{p,\ep}(t,\mathbf y)|}{|\mathbf x- \mathbf y|} d \mathbf y \\
& \leq C\|\bar\omega_{p,\ep}\|_{L^\infty}^{1/2}\|\bar\omega_{p,\ep}\|_{L^2}^{1/2} \| |\mathbf{x}|^{-1}\|_{L^{4/3}(B(0,1))} + C\|\bar\omega_{p,\ep}\|_{L^1}\\
&\leq C\ep^{1/2} + C\|\bar\omega_{p,\ep}\|_{L^1},
\end{split}
\end{equation}
where $\|\bar\omega_{p,\ep}\|_{L^1}$ can be controlled by
\be
\label{control_l1}
\|\bar\omega_{p,\ep}(t,\cdot)\|_{L^1} \leq C  \|\bar\omega_{p,\ep}(t,\cdot)\|_{L^2}^{1/2} (|\supp\bar{\omega}_0| + |\supp\omega_L|)^{1/2} \leq C\ep^{1/2}.
\ee
This finishes the proof of \eqref{u_infty_bd2}.

\medskip
\textbf{Step 3. $p(t)$ has small speed.}
To estimate $\dot p(t)$,  we apply time derivative to the equation $H(t,p(t))=0$ to obtain
\[
\dot p(t) = -\frac{\partial_t H(t,p(t))}{\partial_p H(t,p(t))}.
\]
In \eqref{derH} we already have  $|\partial_p H(t,p(t))| \geq \frac{1}{2}m_L$. So to show $| \dot p(t)|\leq C\ep^{1/2}$, it suffices to show  $|\partial_t H(t,p(t))|\leq C\ep^{1/2}$. A simple computation gives
\begin{equation}\label{Ht}
\begin{split}
\partial_t H(t,p(t)) &=  \int_{\RR_+^2} -  (\bar{\mathbf{u}}(t,\mathbf x) - \mathbf{e}_1) \cdot \nabla \bar\omega(t,\mathbf x)   \, g( x_1 - p(t) ) d\mathbf x\\
&= \int_{\RR_+^2} \bar \omega(t,\mathbf x)  \,(\bar{\textbf{u}}(t,\mathbf x)\cdot\mathbf{e}_1 - 1)   \, g'( x_1 - p(t) ) d\mathbf x.
\end{split}
\end{equation}

Since $\omega_L$ is a steady solution to \eqref{Euler eq: moving frame 1}, we have $(\mathbf{u}_L -\mathbf e_1) \cdot \na \omega_L=0$, thus
\[
\begin{split}
0 &= -\int_{\RR_+^2} (\mathbf{u}_L( \mathbf x) -\mathbf e_1) \cdot \na \omega_L(\mathbf x) g( x_1) d\mathbf x \\
&= -\int_{\RR_+^2} (\mathbf{u}_L( \mathbf x-p(t)\mathbf{e}_1) -\mathbf e_1) \cdot \na \omega_L(\mathbf x-p(t)\mathbf{e}_1) g( x_1 - p(t)) d\mathbf x \\
&=  \int_{\RR_+^2}  \omega_L(\mathbf x-p(t)\mathbf{e}_1) \, (\mathbf{u}_L( \mathbf x-p(t)\mathbf{e}_1)\cdot\mathbf{e}_1 - 1)\, g'( x_1 - p(t)) d\mathbf x .
\end{split}
\]
Subtracting the above from \eqref{Ht} and using the definition of $\bar\omega_{p,\ep}$ in \eqref{def_w_p_ep}, we have
\[
\begin{split}
\partial_t H(t,p(t)) &= \int_{\RR_+^2} \bar\omega_{p,\ep}(t,\mathbf{x})\,  (\bar{\textbf{u}}(t,\mathbf x)\cdot\mathbf{e}_1 - 1)\, g'(x_1-p(t)) d\mathbf{x} \\
& \quad + \int_{\RR_+^2}  \omega_L(\mathbf x-p(t)\mathbf{e}_1) \,(\bar{\mathbf{u}}_{p,\ep}(t,\mathbf{x})\cdot \mathbf{e}_1) \,g'(x_1-p(t)) d\mathbf{x}\\
& =: I_1 + I_2.
\end{split}
\]

Since both $g'$ and $\sup_{t\geq 0}\|\bar{\mathbf{u}}(t,\cdot)\|_{L^\infty}$ are bounded (recall that $\bar\omega_0\in L^\infty \cap L^1$), we bound $I_1$ as
$
|I_1| \leq  C \|\bar\omega_{p,\ep}(t,\cdot)\|_{L^1} \leq C\ep^{1/2},
$
where the last inequality follows from \eqref{control_l1}. 
To control $I_2$, combining $|g'|\leq 3$,  $\omega_L \in L^1$ and \eqref{u_infty_bd2} directly leads to $|I_2|\leq C\ep^{1/2}$.  Finally, putting together the estimates for $I_1$ and $I_2$ yields $
|\partial_t H(t,p(t))|\leq C\ep^{1/2}$, which implies $|\dot p|\leq C\ep^{1/2}$.
\end{proof}

The proof of the above proposition is the main difference between the $\mathbb{T}^{2}$ and $\mathbb{R}^2$ case. The rest of the proof is mainly parallel to those in Section~\ref{subsection_moving_frame}--\ref{sec_torus_growth}, which we include below for the sake of completeness.

\subsection{Evolution in the moving frame centered at $(p(t)+t)\mathbf{e}_1$} If $\bar\omega_0$ satisfies \eqref{initial_r2}, with $p(t)$  defined by \eqref{R2: center def}, we introduce a new moving frame 
\be
\label{eq_rho2}
\rho(t,\cdot) := \bar\omega(t, \cdot+p(t)\mathbf{e}_1) = \omega(t, \cdot+(p(t)+t)\mathbf{e}_1) 
\quad\text{ for } t\geq 0.
\ee
 Since $\omega$ satisfies \eqref{eq vorticity form: Euler eq}, one can check that $\rho$ satisfies the transport equation
\begin{equation}\label{eq_trans3}
\partial_t \rho + \mathbf{v}\cdot \nabla \rho=0
\end{equation}
 with initial data $\rho_0 = \omega_0(\cdot+p(0)\mathbf{e}_1)$, where the velocity field $\mathbf{v}(t,\cdot)$ is given by
 \begin{equation}
 \label{def_v2}
 \begin{split}
 \mathbf{v}(t,\cdot) &= \nabla^{\perp}\Delta^{-1} \rho(t,\cdot) -(\dot p(t)+1)\mathbf{e}_1\\
 &= (\mathbf{u}_L(t,\cdot) - \mathbf{e}_1) + \nabla^{\perp}\Delta^{-1}(\rho(t,\cdot)-\omega_L) - \dot p(t)\mathbf{e}_1\\
 &= (\mathbf{u}_L(t,\cdot) - \mathbf{e}_1) + \bar{\mathbf{u}}_{p,\ep}(t,\cdot +p(t)\mathbf{e}_1) - \dot p(t)\mathbf{e}_1,
 \end{split}
  \end{equation}
where $\bar{\mathbf{u}}_{p,\ep}$ is defined in \eqref{def_upe}.

    \begin{corollary}\label{cor_v2}
    There exists a universal constant $\delta_4  \in (0,\frac{1}{100}m_L]$, such that for any initial data $\omega_0\in \calX_{odd,+}$ satisfying
\be
\label{initial_r2}
|\supp\omega_0| \leq 5, \quad \|  \omega_0 \|_{L^\infty} \le \| \omega_L \|_{L^\infty} + 3  \quad \text{and} \quad \|  \omega_0  -\omega_L\|_{\calX}  \le \delta_4,
\ee
the functions $p, \rho$ and $\mathbf{v}$ given by \eqref{R2: center def}, \eqref{eq_rho2} and \eqref{def_v2} respectively satisfy the following.
\begin{enumerate}[label=\textup{(\alph*)}] \item On the sides $\Gamma_1,\cdots\Gamma_4$ of the square $Q$ in Lemma~\ref{lemma: saddle pts of Lamb dipole}, we have
\begin{equation}\label{goal_v_flux2}
\mathbf{v}(t,\mathbf{x})\cdot \mathbf{n} > \frac{1}{2}c_0 ~~\text{ on }\Gamma_1\cup\Gamma_3 \quad \text{ and } \quad \mathbf{v}(t,\mathbf{x})\cdot \mathbf{n} < -\frac{1}{2}c_0 ~~\text{ on }\Gamma_2\cup\Gamma_4,
\end{equation}
where $\mathbf{n}$ is the outer normal of $Q$, and $c_0$ is as in Lemma~\ref{lemma: saddle pts of Lamb dipole}.

\smallskip
\item $|p(0)| < \frac{1}{10}$.
\end{enumerate}
    \end{corollary}

  \begin{proof}  To prove (a),  let us define $\ep_2 := \min\{\frac{1}{100}m_L, \frac{1}{4}c_0 C_3^{-1}, (\frac{1}{4}c_0 C_3^{-1})^2  \}$,
    where $c_0$ is given in Lemma~\ref{lemma: saddle pts of Lamb dipole}, and $m_L, C_3$ are given in Proposition~\ref{lemma: mass center R}. 
    Applying Proposition~\ref{lemma: mass center R} to $\ep=\ep_2$,  we know there exists a $\delta_4 \in (0,\frac{1}{100}m_L]$ (which is a universal constant since $\ep_2$ is fixed), such that for all $\omega_0\in \calX_{odd,+}$ satisfying \eqref{initial_r2}, we have \eqref{u_infty_bd2} and \eqref{mass center: R 2} hold for all $t\geq 0$.
    
   Using $\ep_2 \leq \frac{1}{4}c_0 C_3^{-1}$, \eqref{u_infty_bd2} yields
    \[
   \| \bar{\mathbf{u}}_{p,\ep}(t,\cdot)\|_{L^\infty} \leq C_3\ep_2 \leq \frac{c_0}{4}.
    \]
Also, combining \eqref{mass center: R 2} with  $\ep_2 \leq (\frac{1}{4}c_0 C_3^{-1})^2$ yields
    \[
    |\dot p(t)| \leq C_3 \sqrt{\ep} \leq \frac{c_0}{4}.
    \]
    Finally, plugging both estimates into \eqref{def_v2} and putting it together with \eqref{u_bdry_ineq}  finishes the proof of (a).

    \medskip
    
    To show (b), we will use the function $H(t,p)$ defined in \eqref{def_H}. Recall that for each fixed $t$, $H$ is non-decreasing in $p$ by \eqref{der_H}. To show that $|p(0)| < \frac{1}{10}$, it suffices to show that $H(0,\frac{1}{10})<0$ and $H(0,-\frac{1}{10})>0$. Let us decompose $H(0,\frac{1}{10})$ as
    \[
    H\Big(0,\frac{1}{10}\Big) = \int_{\RR^2_+} (\omega_0 - \omega_L) \,g\Big(x_1-\frac{1}{10}\Big) d\mathbf{x} +  \int_{\RR^2_+} \omega_L \, g\Big(x_1-\frac{1}{10}\Big) d\mathbf{x} =: T_1 + T_2.
    \]
    Note that $T_2 = -\frac{1}{10} m_L$ by \eqref{HL}. To control $T_1$, we have
    \[
    |T_1| \leq \|g\|_{L^\infty} \|\omega_0-\omega_L\|_{L^2} (|\supp\omega_0| + |\supp\omega_L|)^{1/2} < 10\delta_4 \leq \frac{1}{10}m_L,
    \]
    where in the last step we used that $\delta_4\leq \frac{1}{100}m_L$. Combining these gives $H(0,\frac{1}{10})<0$, and $H(0,-\frac{1}{10})>0$ follows from an identical argument. These imply $|p(0)|<\frac{1}{10}$.
             \end{proof}

Now we are ready to prove Theorem~\ref{thm: small scale R2}.

\begin{proof}[\textbf{\textup{Proof of Theorem~\ref{thm: small scale R2}}}]
Given  $\ep>0$, we construct $\tilde\omega_0 \in C^\infty(\RR^2)\cap \calX_{odd,+}$ such that the following holds, where $\delta_4>0$ and $\eta\in(0,\frac{1}{10})$ are universal constants given by Corollary~\ref{cor_v2} and Lemma~\ref{lemma: saddle pts of Lamb dipole} respectively:

\smallskip
\begin{enumerate}[label=(\alph*)]
\item
 $\|\tilde\omega_0 - \omega_L\|_{\calX} < \min\{\ep, \frac{1}{2}\delta_4\}$.

\smallskip
\item $\tilde\omega_0 = 2$ on $L_1:= \{\mathbf{x}\in\mathbb{R}^2: -3/2\leq x_1\leq -1/2, x_2=\eta/2\}$, and\\
  $\tilde\omega_0=-2$ on $L_2:= \{\mathbf{x}\in\mathbb{R}^2: -3/2\leq x_1\leq -1/2, x_2=-\eta/2\}$.

\smallskip
\item $|\supp \tilde\omega_0| \leq 4$, \quad$\| \tilde\omega_0 \|_{L^\infty(\TT^2)} \le  \| \omega_L \|_{L^\infty} + 2.$

\end{enumerate}
\medskip

To construct $\tilde\omega_0$, we first mollify $\omega_L$ to make it $C^\infty$ (note that the difference in $\|\cdot\|_\calX$ norm can be made arbitrarily small after the mollification), then modify it in two sufficiently small neighborhoods containing $L_1$ and $L_2$, while maintaining its odd-in-$x_2$ property.

Once $\tilde\omega_0$ is constructed, let us consider any $\omega_0 \in C^1(\mathbb{R}^2)\cap \calX_{odd,+}$ satisfying 
\[
|\supp\omega_0| \leq 5,\quad  \|\omega_0 - \tilde\omega_0\|_{L^\infty(\mathbb{T}^2)} \leq 1 \quad\text{ and }\|\omega_0 - \tilde\omega_0\|_{\calX} \leq \frac{1}{2}\delta_4.
\]
Combining it with (a) yields
$
\|\omega_0 - \omega_L\|_{\calX} \leq \delta_4.
$
We can then apply Corollary~\ref{cor_v2}(a) to conclude that \eqref{goal_v_flux2} holds for the velocity field $\mathbf{v}(t,\cdot)$ in the moving frame. As a result, $\mathbf{v}(t,\cdot)$ satisfies Condition 1 in Appendix~\ref{appendix A} for all $t\geq 0$.

Next we will verify that $\rho_0=\rho(0,\cdot) = \omega_0(\cdot+p(0)\mathbf{e}_1)$ satisfies Condition 2 in Appendix~\ref{appendix A}. By Property (b) of $\tilde\omega_0$ and the assumption that $\|\omega_0-\tilde\omega_0\|_{L^\infty}\leq 1$, we have $ \omega_0 \geq 1$ on $L_1$ and $ \omega_0 \leq -1$ on $L_2$. Since $\rho_0=\omega_0(\cdot+p(0)\mathbf{e}_1)$ and $|p(0)|\leq \frac{1}{10}$ by Corollary~\ref{cor_v2}(c), the above implies 
\[\rho_0 \geq 1 \text{ on } \{ -1.4\leq x_1\leq -0.6, x_2=\eta/2\}, \quad \rho_0 \leq -1 \text{ on } \{ -1.4\leq x_1\leq -0.6, x_2=-\eta/2\}.
\]
This shows that $\rho_0$ satisfies Condition 2 in Appendix~\ref{appendix A} on the curves $\mathcal{C}_1:=\{|x_1+1|\leq\eta, x_2=\eta/2\}$ and $\mathcal{C}_1:=\{|x_1+1|\leq\eta, x_2=-\eta/2\}$, where we used that $\eta<1/10$. 

Finally, we apply the superlinear gradient growth result in Lemma~\ref{grad_growth} to the transport equation \eqref{eq_trans3}  to conclude 
$
\int_0^\infty \|\nabla\rho(t,\cdot)\|_{L^\infty}^{-1} dt < C,
$
which finishes the proof since $\|\nabla\omega(t,\cdot)\|_{L^\infty}=\|\nabla\rho(t,\cdot)\|_{L^\infty}$ for all $t\geq 0$.
\end{proof}

\appendix

\section{Superlinear growth under flux assumptions}
\label{appendix A}
In this section, we consider a global-in-time $C^1$ solution $\rho(t,\mathbf{x})$ to the transport equation
\begin{equation}\label{eq_trans2}
\partial_t \rho + \mathbf{v}\cdot \nabla \rho=0 \quad\text{ for }t\geq 0, x\in\Omega,
\end{equation}
where $\mathbf{v}$ is a time-dependent $C^1$ divergence-free velocity field, and the spatial domain $\Omega$ can be either $\mathbb{T}^2$ or $\mathbb{R}^2$. 

%
%

We assume the flow $\mathbf{v}$ and initial condition $\rho_0$ satisfy the following two conditions respectively:

\medskip
\noindent\textbf{Condition 1.} There exists a constant $c_0>0$ and an open parallelogram domain $D\subset\Omega$ with its four sides\footnote{Here the four sides are named in clockwise order. Also, we define each of them as an open line segment, where the vertices are not included.} denoted by $\Gamma_1,\dots,\Gamma_4$, such that 
\[
\mathbf{v}\cdot \mathbf{n} > c_0 ~~\text{ on }\Gamma_1\cup\Gamma_3 \quad \text{ and } \quad \mathbf{v}\cdot \mathbf{n} \leq -c_0 ~~\text{ on }\Gamma_2\cup\Gamma_4
\]
for all $t\geq 0$, where $\mathbf{n}$ is the outer normal of $\partial D$.

\medskip

\noindent\textbf{Condition 2.} Assume there is a constant $K$ and  two (disjoint) curves $\mathcal{C}_1$, $\mathcal{C}_2 \subset \overline{D}$, where each of them has one endpoint on $\Gamma_1$, the other endpoint on $\Gamma_3$, and the rest  inside $D$; and $\rho_0$ satisfies $\rho_0\geq K+1$ on $\mathcal{C}_1$ and $\rho_0\leq K$ on $\mathcal{C}_2$.

\medskip
Under these conditions, we will show $\|\nabla\rho(t,\cdot)\|_{L^\infty(\Omega)}$ must have superlinear growth as time goes to infinity. The lemma and its proof ideas are essentially contained in \cite{denisov2009}, except that \cite{denisov2009} states it as a rectangle instead of a parallelogram, and also uses a discrete time sequence for the proof. For the sake of completeness, we give a quick proof below.

\begin{lemma}\label{grad_growth}
Assume $\rho$ is a $C^1$ solution to \eqref{eq_trans2} with initial condition $\rho_0\in C^1(\Omega)$, where $\mathbf{v}\in C^1([0,\infty)\times\Omega)$ is a time-dependent  divergence-free velocity field. Assume $\mathbf{v}$ and $\rho_0$ satisfy Condition 1 and 2. Then we have  
\[
\int_0^\infty \|\nabla\rho(t,\cdot)\|_{L^\infty(\Omega)}^{-1} dt < 30|D| c_0^{-1},
\]
where $D$ and $c_0$ are given in Condition 1.
\end{lemma}

\begin{proof}
Let $R_0\subset \overline{D}$ be the simply-connected closed region bounded between $\mathcal{C}_1, \Gamma_1, \mathcal{C}_2$ and $\Gamma_3$. We will introduce a smooth cut-off of $\rho_0$, so we can focus on the particles originated from $R_0$. Let $f:\mathbb{R}\to[0,3]$ be a smooth function that satisfies $|f'|<10$ in $\mathbb{R}$, and
\[
f(s)=\begin{cases} 0 & \text{ for } s \in (-\infty, K] \cup [K+1,+\infty),\\
1 &  \text{ for }  s = K+\frac13,\\
2 & \text{ for } s = K+\frac23.
\end{cases}
\]
We then introduce $\mu_0:\overline{D}\to\mathbb{R}$ given by
\begin{equation}\label{temp_init}
\mu_0 (x) := f(\rho_0(x)) 1_{R_0}(x).
\end{equation}
Such definition gives $\mu_0 \in C^1(\overline{D})$, $0\leq \mu_0\leq 3$, and $\supp\mu_0\subset R_0$. 

By Condition 2 and the fact that $\rho_0$ is continuous, we know that the set $\{\mu_0=1\}$ contains $\{\rho_0=K+\frac13\}\cap\overline{D}$, thus has a connected component in $\overline{D}$  that touches both $\Gamma_1$ and $\Gamma_3$. The same can be said for $\{\mu_0=2\}$.

Let us consider the following transport equation in the closed parallelogram domain $\overline D$:
\begin{equation}\label{eq_mu}
\begin{cases}
\partial_t \mu + \mathbf{v}\cdot\nabla\mu=0 &\text{in}\quad [0,\infty)\times D,\\
\mu(0,\cdot)=\mu_0 & \text{in}\quad  \overline D,\\
\mu(t,\cdot)=0 & \text{on} \quad \Gamma_2 \cup \Gamma_4.
\end{cases}
\end{equation}
It is easy to see that \eqref{eq_mu} has a solution $\mu\in C^1([0,\infty)\times\overline{D})$, since $\mathbf{v}, \mu_0$ are both $C^1$, and the zero boundary condition on $\Gamma_2\cup \Gamma_4$ is compatible with the initial data $\mu_0$.\footnote{Note that there is no need to impose boundary conditions on the sides $\Gamma_1$ and $\Gamma_3$ with outgoing flux.}

In addition, the facts that $\mathbf{v}\in C^1$ and the outgoing flux condition on $\Gamma_1 \cap \Gamma_3$ ensures that for all $t\geq 0$, $\{\mu(t,\cdot)=1\}$ has a connected component that touches both $\Gamma_1$ and $\Gamma_3$. The same can be said for $\{\mu(t,\cdot)=2\}$. In particular, these imply that  for all $t\geq 0$, $\mu(t,\cdot)$ attains both values 1 and 2 on $\Gamma_1$.

Let us define by $m(t)$ the mass of $\mu$ inside $D$ at time $t$, i.e.
\[
m(t) := \int_{D}\mu(t,\mathbf{x}) d\mathbf{x}.
\]
Taking the time derivative and applying divergence theorem yields
\[
\begin{split}\frac{d}{dt}m(t) &= -\int_{\Gamma_1} \mu(t,\mathbf{x}) \mathbf{v}(t,\mathbf{x})\cdot\mathbf{n} d\mathbf{x}  -\int_{\Gamma_3} \mu(t,\mathbf{x}) \mathbf{v}(t,\mathbf{x})\cdot\mathbf{n} d\mathbf{x}\\
&\leq -c_0\int_{\Gamma_1} \mu(t,\mathbf{x}) d\mathbf{x},
\end{split}
\]
where we used the zero boundary condition of $\mu$ on $\Gamma_2$ and $\Gamma_4$ in the identity, and Condition 1 in the inequality. This implies
\[
\begin{split}
m(0)-m(t) &\geq c_0 \int_0^t \int_{\Gamma_1} \mu(s,\mathbf{x}) d\mathbf{x} ds\\
& \geq c_0 \int_0^t |\Gamma_1 \cap \{1\leq \mu(s,\cdot)\leq 2\}| ds\\
&\geq c_0 \int_0^t \|\nabla\mu(s,\cdot)\|_{L^\infty(\Gamma_1)}^{-1} ds,\\
\end{split}
\]
where the last inequality follows from the fact that for each $s\geq 0$, there exists some line segment $L(s)$ on $\Gamma_1$ such that $1\leq \mu(s,\cdot)\leq 2$ on $L(s)$, and $\mu(s,\cdot)$ takes value 1 and 2 on its endpoints. Mean value theorem then yields $L(s) \|\nabla\mu(s)\|_{L^\infty} \geq 1$.
 
In the above inequality, note that $m(t)\geq 0$, and $m(0)\leq 3|D|$ since $0\leq \mu_0\leq 3$ in $D$. Sending $t\to\infty$, we have
\begin{equation}
\label{eq_int_bd}
 \int_0^\infty  \|\nabla\mu(t,\cdot)\|_{L^\infty(\Gamma_1)}^{-1} dt \leq 3|D| c_0^{-1}.
\end{equation}

Finally, we make the following connection between $\mu$ in \eqref{eq_mu} and $\rho$ in \eqref{eq_trans2}: note that they are transported by the same velocity field $\mathbf{v}$ in $D$, and their initial conditions $\mu_0$ and $\rho_0$ are related by \eqref{temp_init}.  As a result, if $\mu(t,\mathbf{x})\neq 0$ for some $t\geq 0, x\in \overline{D}$, we must have $\mu(t,\mathbf{x})=f(\rho(t,\mathbf{x}))$. This leads to the following (where we also use $|f'|\leq 10$): \[
 \|\nabla\mu(t,\cdot)\|_{L^\infty(\overline{D})}\leq   \|\nabla (f\circ \rho)(t,\cdot)\|_{L^\infty(\overline{D})}\leq  10 \|\nabla \rho(t,\cdot)\|_{L^\infty(\overline{D})},
\]
and plugging it into \eqref{eq_int_bd} gives
\[
 \int_0^\infty  \|\nabla\rho(t,\cdot)\|_{L^\infty(\Omega)}^{-1} dt \leq 30|D| c_0^{-1},
\]
finishing the proof.
\end{proof}

\bibliographystyle{siam}
\bibliography{Bib_1}

\end{document}